\documentclass[
 hidempi, 
]{mpi2015-cscpreprint}

\usepackage{graphicx} % Required for inserting images
\usepackage{longtable}
\usepackage{lscape}
\usepackage{a4wide}
\usepackage{tabularray}
\usepackage{xcolor}
\usepackage{here}
\usepackage{amsmath}
\usepackage{amssymb}
\usepackage{natbib}
\usepackage{pgffor}
\usepackage{color}
\usepackage{cleveref}
\usepackage{authblk}
\usepackage{subcaption}
\usepackage{multirow}
\usepackage{arydshln}
\usepackage{graphbox}
\usepackage{colortbl}
\usepackage{amsfonts}
\usepackage{epsfig}
\usepackage{colortbl}
\usepackage{comment}
\usepackage{mathdots}
\usepackage{mathrsfs}
\usepackage{epstopdf}
\usepackage{makeidx} 
\usepackage{ifpdf}
\usepackage{arydshln}
\usepackage{tikz,tikz-3dplot}
\usetikzlibrary{arrows,decorations.pathmorphing,backgrounds,fit,positioning,shapes.symbols,chains,mindmap,trees,calc}
\usepackage{xfp}
\usepackage{cancel}

\definecolor{codegreen}{rgb}{0,0.6,0}
\definecolor{codegray}{rgb}{0.5,0.5,0.5}
\definecolor{codepurple}{rgb}{0.58,0,0.82}
\definecolor{backcolour}{rgb}{0.95,0.95,0.92}
\title{Variable decoupling and the Kolmogorov Superposition Theorem for rational functions}

\author[$\ast$]{A. C. Antoulas}
\affil[$\ast$]{Department of Electrical and Computer Engineering, Rice University, Houston, TX, USA.\authorcr
	\email{aca@rice.edu}}

\author[$\dag$]{I. V. Gosea}
\affil[$\dag$]{Max Planck Institute for Dynamics of Complex Technical Systems, CSC Group \\
	Sandtorstr. 1, 39106 Magdeburg, Germany.\authorcr
	\email{gosea@mpi-magdeburg.mpg.de}, \orcid{0000-0003-3580-4116}}

\author[$\S$]{C. Poussot-Vassal}
\affil[$\S$]{DTIS, ONERA, Universit\'e de Toulouse, 31000, Toulouse, France.\authorcr
	\email{charles.poussot-vassal@onera.fr}, \orcid{0000-0001-9106-1893}}

\abstract{The goal of this paper is to show that for rational multivariate functions, the Kolmogorov Superposition Theorem {(KST)}  involves several univariate functions, which can be written down by inspection. In other words, {no computation} is required for {decoupling the variables of multivariate rational functions}.  The key tool for this development is the Loewner Framework for multivariate functions. Applications of this result involve the approximation of multivariate non-rational functions by low-complexity multivariate rational and polynomial functions.
}


\novelty{We first show that the Loewner matrix null space can be written down by inspection, i.e. with no computation. Second, we demonstrate how variable decoupling for rational multivariate functions can be achieved. This result complements the Kolmogorov Superposition Theorem, which addresses the issue for arbitrary continuous functions. We show that this decoupling can be achieved by means of simple function evaluations, i.e without computation.}

\definecolor{blue}{RGB}{16,97,169} % bleu onera
\definecolor{grey}{RGB}{88, 102, 110} % gris onera 
\definecolor{green}{RGB}{65,209,204} % turquoise
\definecolor{orange}{RGB}{224, 131, 0} % corail
\definecolor{pink}{RGB}{255,105,180} % rose
\definecolor{black}{RGB}{0, 0, 0} % corail

\newfont{\Bb}{msbm10 scaled\magstep1}

\def\IC{{\mathbb C}}

\def\IH{{\mathbb H}}

\def\IL{{\mathbb L}}
\def\IN{{\mathbb N}}
\def\IR{{\mathbb R}}

\newtheorem{theorem}{Theorem}[section]
\newtheorem{lemma}{Lemma}[section]

\newtheorem{proposition}{Proposition}[section]
\newtheorem{corollary}{Corollary}[section]
\newtheorem{remark}{Remark}[section]

\newtheorem{proof}{Proof}[section]

\newcommand{\bB}{{\mathbf B}}

\newcommand{\bH}{{\mathbf H}}
\newcommand{\bI}{{\mathbf I}}

\newcommand{\bL}{{\mathbf L}}

\newcommand{\bS}{{\mathbf S}}
\newcommand{\bT}{{\mathbf T}}

\newcommand{\bV}{{\mathbf V}}

\newcommand{\bX}{{\mathbf X}}

\newcommand{\bZ}{{\mathbf Z}}

\newcommand{\bd}{{\mathbf d}}

\newcommand{\bff}{{\mathbf f}}
\newcommand{\bg}{{\mathbf g}}

\newcommand{\bl}{{\mathbf l}}

\newcommand{\bn}{{\mathbf n}}
\newcommand{\bp}{{\mathbf p}}
\newcommand{\bq}{{\mathbf q}}
\newcommand{\br}{{\mathbf r}}
\newcommand{\bs}{{\mathbf s}}
\newcommand{\bt}{{\mathbf t}}

\newcommand{\bv}{{\mathbf v}}
\newcommand{\bw}{{\mathbf w}}
\newcommand{\bx}{{\mathbf x}}
\newcommand{\by}{{\mathbf y}}
\newcommand{\bz}{{\mathbf z}}

\def\IC{{\mathbb C}}

\def\IR{{\mathbb R}}

\def\IL{{\mathbb L}}
\def\IN{{\mathbb N}}

\newcommand{\bTheta}{{\boldsymbol{\Theta}}}

\newcommand{\bDelta}{{\boldsymbol{\Delta}}}
\newcommand{\bPhi}{{\boldsymbol{\Phi}}}
\newcommand{\bPsi}{{\boldsymbol{\Psi}}}
\newcommand{\bOmega}{{\boldsymbol{\Omega}}}
\newcommand{\bomega}{{\boldsymbol{\omega}}}
\newcommand{\bphi}{{\boldsymbol{\phi}}}

\newcommand{\bpsi}{{\boldsymbol{\psi}}}

\newcommand{\balpha}{{\boldsymbol{\alpha}}}
\newcommand{\bell}{{\boldsymbol{\ell}}}
\newcommand{\bgamma}{{\boldsymbol{\gamma}}}

\newcommand{\bone}{{\mathbf 1}}

\providecommand{\col}[1]{{{#1}}}

\newcommand{\sysord}[0]{{\nu}}
\newcommand{\sysordi}[1]{{\sysord_{#1}}}
\newcommand{\ord}[0]{{n}}

\newcommand{\pare}[1]{\left( #1\right)}
\newcommand{\vargen}[2]{{#1}_{#2}}
\newcommand{\var}[1]{\vargen{x}{#1}}

\newcommand{\lani}[2]{\col{\vargen{\lambda}{#1}({{#2}})}}
\newcommand{\ladim}[0]{{k}}
\newcommand{\ladimi}[1]{{\ladim_{#1}}}

\newcommand{\mudim}[0]{{q}}

\newcommand{\poly}[0]{\bp}
\newcommand{\num}[0]{\bn}
\newcommand{\den}[0]{\bd}
\newcommand{\Blag}[1]{\bB_{\textrm{lag}}(#1)}
\newcommand{\lag}[3]{\bl_{#3}^{#2}(#1)}
\newcommand{\lagn}[3]{\bell_{#3}^{#2}(#1)}
\newcommand{\barynum}[1]{{\textbf{Bary}}^{\textrm{num}}_{#1}}
\newcommand{\baryden}[1]{{\textbf{Bary}}^{\textrm{den}}_{#1}}

\newcommand{\vars}[0]{s}
\newcommand{\varS}[0]{\bs}
\newcommand{\varsi}[1]{{\vars_{#1}}}
\newcommand{\vart}[0]{t}
\newcommand{\varT}[0]{\bt}
\newcommand{\varti}[1]{{\vart_{#1}}}
\newcommand{\varz}[0]{z}
\newcommand{\varZ}[0]{\bz}
\newcommand{\varzi}[1]{{\varz_{#1}}}
\newcommand{\varx}[0]{x}
\newcommand{\varX}[0]{\bx}
\newcommand{\varxi}[1]{{\varx_{#1}}}
\newcommand{\vary}[0]{y}
\newcommand{\varY}[0]{\by}
\newcommand{\varyi}[1]{{\vary_{#1}}}

\begin{document}

\maketitle

%%%%%%%%%%%%%%%%%%%%%%%%%%%%%%%%%%%%%%%
\addtocontents{toc}{\setcounter{tocdepth}{2}} 
\tableofcontents

%%%%%%%%%%%%%%%%%%%%%%%%%%%%%%%%%
%%%%%%%%%%%%%%%%%%%%%%%%%%%%%%%%%%%%%
\section{Introduction} 

\subsection{Background}

Hilbert's 13th problem (1900) asks the question:  are there any genuine continuous multivariate functions? Hilbert conjectured the existence of a three-variable continuous function which cannot be expressed in terms of the composition and addition of two-variable continuous functions. For a detailed exposition of the issues related to  Hilbert's 13th problem, we refer the reader to \cite{morris}. 

Subsequently, in problem \#119 of the book \cite{polya}, Gy\"orgy P\'olya and G\'abor Szeg\"o, ask the question: are there actually functions of three variables? Stated differently: is it possible to use compositions of functions of two or fewer variables to express any function of three variables? In other words, do there exist two-variable functions $\phi$ and $\psi$ such that any three-variable function $\bff$ can be expressed as:  $\bff(\varX,\varY,\varZ)= \phi(\psi(\varX,\varY),\varZ)$. Subsequently, they prove that the  function $\bff(\varX,\varY,\varZ)=\varX\varY+\varX\varZ+\varY\varZ$, does {not have this property.

The  Kolmogorov Superposition Theorem (1957), shortly KST, answers the Hilbert conjecture negatively.  It shows that continuous functions of several variables can be expressed as the composition and superposition of functions of one variable. Consequently, there are no true functions of three variables. A recent form of this result, with contributions from various researchers over the years, including Kolmogorov, Arnol'd, Kahane, Lorenz, and Sprecher, is discussed in \cite{morris}. It is stated that given a continuous function $\bff: [0,1]^3\rightarrow \IR$ of  three variables, there exist real numbers $\lambda_i$, $i=1,2$, seven  single-variable continuous functions ${\bphi_k}: [0,1]\rightarrow\IR$, $k=1,\cdots,7$, called inner functions, and a single-variable
function ${\bg}: \IR\rightarrow\IR$, called  outer function, such that
\begin{equation}
\label{eq:kst-original}
\bff(\varX,\varY,\varZ)=\sum_{k=1}^7 {\bg} ({\bphi_k}(\varX)+\lambda_1{\bphi_k}(\varY)+ \lambda_2{\bphi_k}(\varZ)),~~
\mbox{for all}~~ (\varX,\varY,\varZ)\in \left[0,1\right]^3.
\end{equation}

In this result, $\lambda_i$ and $\bphi_k$ do not depend on $\bff$. Thus, for $\ord=3$,  eight functions are needed together with two real scalars $\lambda_i$. Typical forms (fractal by nature) can be found in  \cite[Figure 1]{lai24} and \cite[Figure 4.3]{actor}.

\begin{remark} KST simplifies dealing with multivariate functions by breaking them down into manageable single-variable components. From a computational point of view, this result is hard to use because the (inner) functions $\bphi_k$ and the (outer) function $\bg$ are fractals and can only be  determined iteratively. To address this issue, methods based on neural networks have emerged as a way to approximate the inner and outer functions using, for instance, splines. For details, we refer the reader to Kolmogorov spline networks in \cite{igelnik2003kolmogorov}, Kolmogorov neural networks in \cite{ismayilova2024kolmogorov}, or Kolmogorov-Arnold networks (KANs) in \cite{kan}.
\end{remark}

\subsection{Contributions and paper structure}

Our aim in this work is to establish connections between the Loewner framework for rational interpolation of multivariate functions and the KST. The main feature of the KST is that it decouples the variables involved in the underlying multivariate function. The main result in our approach shows that for rational multivariate functions, variable decoupling can be achieved by means of simple  function evaluation.  As a result, no computation is required. This is to be compared with the iterative  fractal function determination required in KST, as quoted in \cite{morris} and analyzed in \cite{actor}.

The result is based on the fact that the space of $1$D-Loewner matrices, associated with the  single-variable rational functions ~$\bH(\varS)={\num(\varS)}/{\den(\varS)}$, depends  exclusively on the denominator $\bd(\varS)$, evaluated at  appropriately chosen (right) interpolation points. Thus, it can be written down  by inspection. Consequently, making use of the results in \cite{AGPV25}\footnote{See also the dedicated webpage: \url{https://cpoussot.github.io/mlf}}, the decoupling of the variables of multivariate rational functions can be achieved by means of simple evaluations of a set of single-variable rational functions associated with $\bH$, and consequently without computation. 

The paper is organized as follows: \cref{sec:lagrange} discusses general arguments related to Lagrange interpolation and the role of Loewner matrices. It shows how, in the single variable case, barycentric coefficients can be written down by inspection and boils down to a separate treatment of the numerator and denominator. In \cref{sec:poly_multi}, this result is extended to polynomials only, with multiple variables, and provides a recursive manner to construct an approximant. Then, \cref{sec:polya-szego} illustrates the result on the three-variate P\'olya-Szeg\"o function. Finally,  \cref{sec:general-case} provides the general variable decoupling solution for fitting multivariate rational functions, involving basically no computation. We conclude in \cref{sec:conclusions}.

\subsection{Notations}

Throughout the paper, the following notations are adopted: $\IN$, $\IR$ and $\IC$ denote,  respectively, the natural, real and complex domains. We respectively denote $\num$ and $\den$  the numerator and denominator polynomials of the rational function $\bH$; each may be an $\ord$-variable, where the number of variables is $\ord\in\IN$. Polynomials are denoted $\poly$. The $l$-th variable is denoted by $\var{l}$ ($l=1,\cdots,\ord$), but for a simpler exposition (up to Section \cref{sec:general-case}), we also employ the variables $\varS$, $\varT$, $\varZ$ and $\varX$ to denote the first, second, third and fourth variables of a polynomial, e.g. $\poly(\varS,\varT,\varZ,\varX)$. $\IC[\varS]$ denotes the set of polynomials in the $\varS$-variable with coefficients in the complex domain, and $\mbox{deg}(\poly)$ is the degree of polynomial $\poly$. $\bone_k$ denotes the $k$-dimensional column vector composed of ones. Symbols $\otimes$ and $\odot$ respectively denote the Kronecker and Hadamard products. In addition, $\sum_{\text{row-wise}}$ notation denotes the vector sum along its rows.

%The computational complexity reported in \,{\bf Table 1}\, is not necessary, because {\bf no computation of null spaces} is necessary.

\section{Lagrange interpolation, barycentric representation and the Loewner matrix}\label{sec:lagrange}

\subsection{The Lagrange interpolation perspective}

Given distinct $\lambda_j\in\IC, \ j=1,\ldots,\ladim$ and $ \lambda_j\neq \lambda_i$, let the basis of polynomials of degree $\ladim-1$ be composed of:
\begin{equation}\label{eq:basis-lag}
\lag{\varS}{\lambda}{j}=\prod_{i=1,i\neq j}^\ladim (\varS-\lambda_i),~ j=1,\ldots,\ladim.
\end{equation}

For given constants $\alpha_j\neq 0$ and the pairs of nodes and points $\{(\lambda_j,w_j) \vert \ j=1,\cdots,\ladim\}$, where $w_j\in\IC$ (also called "data"),  the barycentric formula discussed in \cite{BT04}  provides the following form of a rational interpolant: 
\begin{equation}\label{eq:lag}
\bg(\varS)=\frac
{\sum_{j=1}^{\ladim} \frac{\alpha_j w_j}{\varS-\lambda_j}}
{\sum_{j=1}^{\ladim} \frac{\alpha_j}{\varS-\lambda_j}} 
= \frac{\sum_{j=1}^\ladim \alpha_j w_j \lag{\varS}{\lambda}{j}}{\sum_{j=1}^\ladim\alpha_j\lag{\varS}{\lambda}{j}}.
\end{equation}

It directly follows that $\bg(\lambda_j)=w_j$, for $j=1,\cdots,\ladim$. We note that \cref{eq:lag} is parameterized by the $\alpha_j$'s. One may choose these parameters to recover additional quantities, e.g., the Lagrange interpolating polynomial (as shown below). 

Let $\bgamma=[\gamma_1,\cdots,\gamma_\ladim]^\top\in\IC^\ladim$ be the  Lagrange coefficients vector, where:
\begin{equation}\label{eq:gamma}
\gamma_j=\dfrac{1}{\prod_{i=1,i \neq j}^\ladim(\lambda_j-\lambda_i)}, ~j=1,\ldots,\ladim. 
\end{equation}

Introduce the normalized Lagrange basis $\lagn{\varS}{\lambda}{j}$, which is determined by scaling the polynomials $\lag{\varS}{\lambda}{j}$ \cref{eq:basis-lag} as, 
\begin{equation}\label{eq:basis-lagn}
\lagn{\varS}{\lambda}{j}=\gamma_j \lag{\varS}{\lambda}{j}=\frac{\lag{\varS}{\lambda}{j}}{\lag{\lambda_j}{\lambda}{j}}=\prod_{i=1,i\neq j}^\ladim\frac{\varS-\lambda_i}{\lambda_j-\lambda_i},~j=1,\ldots,\ladim.
\end{equation}

Given the normalized Lagrange basis \cref{eq:basis-lagn} and $\alpha_{j}\in\IC$, $j=1,\cdots,\ladim$, let $\poly(\varS)=\sum_{j=1}^\ladim w_{j}\lagn{\varS}{\lambda}{j}$ be the Lagrange interpolating polynomial of degree less than $\ladim$ which satisfies the interpolation conditions $\pi(\lambda_j)=w_{j}$. By choosing $\alpha_j = \gamma_j$, noting $\sum_{j=1}^\ladim \lagn{\varS}{\lambda}{j} = 1$, and substituting these values into \cref{eq:lag}, one can write the barycentric form as:
\begin{equation}\label{eq:lag2poly}
\bg_{\textrm{poly}}(\varS)= \frac{\sum_{j=1}^\ladim   \gamma_j w_j\lag{\varS}{\lambda}{j}}{\sum_{j=1}^\ladim\gamma_j\lag{\varS}{\lambda}{j}}
= \frac{\sum_{j=1}^\ladim w_j \lagn{\varS}{\lambda}{j}}{\sum_{j=1}^\ladim \lagn{\varS}{\lambda}{j}}=  \sum_{j=1}^\ladim w_j \lagn{\varS}{\lambda}{j} = \poly(\varS).%, \ \ \ \text{since} \ \sum_{j=1}^\ladim \lagn{\varS}{\lambda}{j} = 1.
\end{equation}

Now, assume that  the data $w_j$ come from the sampling of the rational function $\bH(\varS) = {\num(\varS)}/{\den(\varS)}$ at the $\lambda_j$'s nodes, i.e., $w_j = \num(\lambda_j)/\den(\lambda_j)$. We also assume that the degrees of $\num(\varS)$ and $\den(\varS)$ are less than or equal to $\ladim-1$. In order to choose the $\alpha_j$ weights so that in \cref{eq:lag} we have $\bg(\varS) = \bH(\varS)$, we must have $\alpha_j = \den(\lambda_j) \gamma_j$, for $j=1,\ldots,\ladim$. By substituting these values into (\cref{eq:lag}), one can write the new barycentric form as:
\begin{align}
\bg_{\textrm{ratio}}(\varS) &= \frac{\sum_{j=1}^\ladim  \den(\lambda_j)  \gamma_j w_j\lag{\varS}{\lambda}{j}}{\sum_{j=1}^\ladim \den(\lambda_j) \gamma_j\lag{\varS}{\lambda}{i}}
= \frac{\sum_{j=1}^\ladim  \den(\lambda_j) \frac{\num(\lambda_j)}{\den(\lambda_j)} \lagn{\varS}{\lambda}{j}}{\sum_{j=1}^\ladim \den(\lambda_j)  \lagn{\varS}{\lambda}{j}}  \nonumber \\
&=  \frac{\sum_{j=1}^\ladim \num(\lambda_j) \lagn{\varS}{\lambda}{j}}{\sum_{j=1}^\ladim \den(\lambda_j)  \lagn{\varS}{\lambda}{j}} = \frac{\num(\varS)}{\den(\varS)} = \bH(\varS), \forall \varS \in \IC.
\label{eq:lag2}
\end{align}

In this derivation, it is easy to note that the polynomial numerator $\sum_{j=1}^\ladim \num(\lambda_j) \lagn{\varS}{\lambda}{j}$ coincides with the polynomial $\num(\varS)$ at all $\lambda_j$'s, i.e., at $\ladim$ nodes. Thus,   $\sum_{j=1}^\ladim \num(\lambda_j) \lagn{\varS}{\lambda}{j} = \num(\varS)$, for all $\varS \in \IC$. Similarly, the same considerations hold for the denominator. Thus,   $\sum_{j=1}^\ladim \den(\lambda_j) \lagn{\varS}{\lambda}{j} = \den(\varS)$, for all $\varS \in \IC$.

%%%%%%%%%%%%%%%%%%%%%%%%%%%%%%%%%%%%%%%%%%%%%%%%%
\subsection{The Loewner matrix perspective}

One can obtain the same result as above using a Loewner matrix approach. This route is motivated by a practical, data-driven scenario in which one does not have explicit access to the denominator values $\den(\lambda_j), \ j=1,\ldots,\ladim$, but only to the pairs $\{(\lambda_j,w_j) \vert \ j=1,\ldots,\ladim \}$ (where $w_j = \num(\lambda_j)/\den(\lambda_j)$).

As before, one can choose free parameters $\alpha_j$ so that additional interpolation constraints are satisfied: $\bg(\mu_i)=v_i$, $i=1,\cdots,\mudim$, where  $\mu_i\neq\mu_j$. For this to hold, one needs to impose the condition $\IL\balpha =0$, where $\IL\in\IC^{\mudim\times \ladim}$ is the Loewner matrix, whose entries are given by 
\begin{equation}
(\IL)_{i,j}=\frac{v_i-w_j}{\mu_i-\lambda_j}, 
\label{eq:loewner-matrix}
\end{equation}
and $\balpha=[\alpha_1,\cdots,\alpha_\ladim]^\top\in\IC^\ladim$, $\bw=[w_1,\cdots,w_k]^\top\IC^\ladim$ and $\bv=[v_1,\cdots,v_k]^\top\IC^\mudim$. Assume that the  pairs $(\lambda_j,w_j)$ and $(\mu_i,v_i)$ are samples of ${\num(\varS)}/{\den(\varS)}$, and form the Loewner matrix  $\IL$ as above. Here, we consider mutually distinct $\mu_i,\lambda_j\in\IC$, $i=1,\cdots,\mudim$, $j=1,\cdots,\ladim$ and $\sysord=\mbox{max}(\mbox{deg}(\num)$, $\mbox{deg}(\den))$.  The following result,  first stated in \cite{AA86}, holds:
\begin{equation} \label{eq:loewner-bezoutian}
\bDelta_\mu \, \IL \,\bDelta_\lambda = \bV_\mu^\top \, \bB(\num,\den)\,  \bV_\lambda ,
\end{equation}
where $\bDelta_\mu$, $\bDelta_\lambda$ are diagonal matrices given by 
\begin{equation}
\bDelta_\mu=\mbox{diag}\left[\den(\mu_1)\cdots\den(\mu_\mudim)\right]\in\IC^{\mudim\times \mudim}~\text{and}~ 
\bDelta_\lambda=\mbox{diag}\left[\den(\lambda_1)\cdots\den(\lambda_\ladim)\right]
\in\IC^{\ladim\times \ladim}, 
\label{eq:loewner-diag}
\end{equation}
$\bV_\mu$, $\bV_\lambda$ are Vandermonde matrices given by 
\begin{equation}
(\bV_\mu)_{i,j}=\mu_j^{i-1}~\text{and} ~(\bV_\lambda)_{i,j}=\lambda_j^{i-1},
\label{eq:loewner-vandermonde}
\end{equation}
and $\bB(\num,\den)\in\IC^{\sysord\times \sysord}$,  is  the Bezoutian
%\footnote{For definition and details, please see \url{https://en.wikipedia.org/wiki/Bezout_matrix}} 
of the polynomials $\num$ and $\den$. For an overview of Bezoutians, we refer the readers to \cite{HeinigRost2010Bezoutians}. It is well known that $\bB(\num,\den)$ is  non-singular if and only if the polynomials $\num$ and $\den$ are coprime \cite{Lerer1995Invertibility}. It readily follows that for $\mudim,\ladim> \sysord$, the rank of the Loewner matrix is also $\sysord$.

\begin{theorem}\label{thm:inspection_den} Let $\bgamma\in\IC^\ladim$, be the vector of Lagrange coefficients associated with the $\lambda_j$'s \cref{eq:gamma}. If $\sysord=\ladim-1$,  the (right) null space of $\bV_\lambda$ \cref{eq:loewner-vandermonde} is spanned by $\bgamma$. Therefore, the (right)  null space of matrix $\IL$ \cref{eq:loewner-matrix} is spanned by the barycentric weights:
\begin{equation}\label{eq:bary_den}
\baryden{\varS}=\bDelta_\lambda\bgamma=
\left[\begin{array}{cccc}
\gamma_1  \bd(\lambda_1) & \gamma_1  \bd(\lambda_2) & \cdots & \gamma_{\ladim} \bd(\lambda_{\ladim})
\end{array}\right]^\top\in\IC^{\ladim}.
\end{equation}
\end{theorem}

\begin{proof}
First, check that $\bgamma$ \cref{eq:gamma} spans the null space of the Vandermonde matrix $\bV_\lambda$; then, following \cref{eq:loewner-bezoutian}, the right null space of $\IL$ is spanned by $\bDelta_\lambda\bgamma$, which concludes the proof.
\end{proof}

\begin{corollary} \label{cor:inspection_H} Similarly, the barycentric weights associated with the numerator are: 
\begin{equation}\label{eq:bary_num}
\barynum{\varS}=
\left[\begin{array}{cccc}
\gamma_1\num(\lambda_1)& \gamma_2 \num(\lambda_2) &\cdots&\gamma_{\ladim}\num(\lambda_{\ladim})
\end{array}\right]^\top\in\IC^{\ladim}.
\end{equation}
Thus the rational function $\bH(\varS)$ can be expressed as:
\begin{equation}\label{eq:bary}
\bH(\varS)=\frac{\sum_{j=1}^\ladim\frac{\gamma_j \num(\lambda_j)}{\varS-\lambda_j}}
{\sum_{j=1}^\ladim\frac{\gamma_j \den(\lambda_j)}{\varS-\lambda_j}} =
\frac{\sum_{j=1}^\ladim(\barynum{\varS}\odot \Blag{\varS})_j}{\sum_{j=1}^\ladim(\baryden{\varS}\odot  \Blag{\varS})_j},
\end{equation}
where
\begin{equation}
\Blag{\varS}=\left[\frac{1}{\varS-\lambda_1}~\cdots~ \frac{1}{\varS-\lambda_\ladim}\right]^\top\,\in\,\IC^\ladim[\varS]
\end{equation}
\end{corollary}

\begin{remark} 
\Cref{thm:inspection_den} and \cref{cor:inspection_H} show  that the barycentric weights of single-variable  rational functions can be written down by inspection; in other  words, no null space computation is required, as long as one has access to the numerator $\num$ and denominator $\den$, independently. This result also shows that the right/left null space of the Loewner matrix  \cref{eq:loewner-matrix}, based on the rational function ${\num(\varS)}/{\den(\varS)}$, depends exclusively on the denominator $\den(\varS)$ and the associated right/left interpolation points. More specifically, $\baryden{\varS}=\balpha$ and  $\barynum{\varS}=\balpha\odot\bw$, where $\IL\balpha =0$ and $\bw=[\bH(\lambda_1),\cdots,\bH(\lambda_\ladim)]^\top$.
\end{remark}

\begin{remark} 
There are several similarities between the above \cref{thm:inspection_den} and a result in realization theory; for details, see \cite[\S 4.4]{acabook}. Roughly speaking, the Loewner framework (with the main tool provided by Loewner matrices $\IL$) deals with data at finite locations in the complex plane, while the Hankel framework (with the main tool provided by Hankel matrices $\IH$) deals with data at infinity (i.e., Markov parameters). In both cases, the null spaces of $\IL$ and $\IH$ encode the information about the solution. Both cases have similarities, namely: (a) the denominator of the solution depends on the null space of the associated Loewner or Hankel matrix, (b) the null space of both Loewner and Hankel matrices is determined by a single polynomial. Thus, the structure of the kernel of Loewner matrices mentioned above parallels the structure of the kernel of Hankel matrices (see \cite{gutknecht} for details), and (c) the null space of the Loewner (resp. Hankel) matrix gives the coefficients of the polynomial in the  Lagrangian (resp. monomial) basis.
\end{remark}

\subsection{Edifying example}

To summarize the contribution given in \cref{thm:inspection_den} and \cref{cor:inspection_H} and standard Loewner ones \cite{AA86}, let us consider the rational function $\bH(\varS)=\num(\varS)/\den(\varS)=(\varS+1
)/({\varS}^2+\varS+2
)$, together with the interpolation points $\lambda=\left[\begin{array}{ccc} 1 & 2 & 3 \end{array}\right]
$ and $\mu=\left[\begin{array}{ccc} \frac{3}{2} & \frac{5}{2} & \frac{7}{2} \end{array}\right]
$ ($j=1,\cdots,k=3
$ and $i=1,\cdots,q=3
$), thus leading to data $\bw=\bH(\lambda_j)=\left[\begin{array}{ccc} \frac{1}{2} & \frac{3}{8} & \frac{2}{7} \end{array}\right]
$ and $\bv=\bH(\mu_i)=\left[\begin{array}{ccc} \frac{10}{23} & \frac{14}{43} & \frac{18}{71} \end{array}\right]
$.

\subsubsection{Interpolation via null space computation \cite{AA86}} 
By constructing the Loewner matrix $\IL$ \cref{eq:loewner-matrix}, its null space $\balpha_1$ is computed and leads to
\begin{equation}
\IL=\left[\begin{array}{ccc} -\frac{3}{23} & -\frac{11}{92} & -\frac{16}{161}\\[1mm] -\frac{5}{43} & -\frac{17}{172} & -\frac{24}{301}\\[1mm] -\frac{7}{71} & -\frac{23}{284} & -\frac{32}{497} \end{array}\right]
 ~\text{and}~\balpha_1=\left[\begin{array}{c} \frac{2}{7}\\[1mm] -\frac{8}{7}\\[1mm] 1 \end{array}\right]
.
\end{equation}
By using \cref{eq:lag}, the function $\bH$ is recovered based on data and null space computation.

\subsubsection{Interpolation via inspection of $\num(\varS)$ and $\den(\varS)$}
By now using \cref{thm:inspection_den}, one computes $\bgamma$ using $\lambda_j$'s interpolation points, $\bDelta_\lambda$ using the evaluation of the denominator, and gets $\baryden{\varS}$ as follows:
\begin{equation}
\bgamma=\left[\begin{array}{c} \frac{1}{2}\\[1mm] -1\\[1mm] \frac{1}{2} \end{array}\right]
~\text{,}~ \bDelta_\lambda=\left[\begin{array}{ccc} 4 & 0 & 0\\[1mm] 0 & 8 & 0\\[1mm] 0 & 0 & 14 \end{array}\right]
~\text{and}~\baryden{\varS}=\left[\begin{array}{c} 2\\[1mm] -8\\[1mm] 7 \end{array}\right]
.
\label{eq:ex1-step1}
\end{equation}
Now using \cref{cor:inspection_H} (and evaluation of numerator), we similarly get 
\begin{equation}
\barynum{\varS}=\left[\begin{array}{c} 1\\[1mm] -3\\[1mm] 2 \end{array}\right]
~\text{and}~\Blag{\varS}=\left[\begin{array}{c} \frac{1}{\varS-1}\\[1mm] \frac{1}{\varS-2}\\[1mm] \frac{1}{\varS-3} \end{array}\right]
.
\label{eq:ex1-step2}
\end{equation}
Then, by using \cref{eq:bary}, the function $\bH$ is recovered. Here, the function is recovered without computation.

\subsubsection{Interpolation via inspection of $\den(\varS)$ and with $\bw$}

Finally, if we consider the access to both the data $\bw$ and the denominator, one can set $\balpha_2:=\baryden{\varS}$ as in \cref{eq:ex1-step1}, and then apply \cref{eq:lag}; the function $\bH$ is recovered. Here again, the function is recovered with no computation.

\subsection{Toward the multivariate case}

From this first result, it is now clear that interpolation can be written down by inspection, given the fact that we have access to some information: (i) either the numerator and denominator or (ii) data and denominator. Now, taking advantage of \cite{AGPV25}, the rest of this paper is dedicated to showing that this holds for multivariate polynomials and hence also rational functions as well, i.e., the barycentric representation of multivariate  rational functions can be written down by inspection. This leads to a decoupling  of the variables of multivariate functions and the separate treatment of numerator and denominator, with consequences (in more general settings)  on the curse of dimensionality (see \cite{AGPV25} for details). This also shows how the \textit{variable decoupling} of multivariate functions, introduced by Kolmogorov, results from simple evaluations in the rational case.

\section{The multivariate polynomial case}\label{sec:poly_multi}

Given the considerations of the previous section, it is clear that we can restrict our attention to the polynomial case. The rational one is further detailed in the appendix.
%Given the Lagrange basis \cref{eq:basis-lagn}, and $\alpha_{j_1}\in\IC$, $j_1=1,\cdots,\ladimi{1}$, $\pi(\varS)=\sum_{j_1=1}^\ladimi{1}\alpha_{j_1}\lagn{\varS}{\vars}{j_1}$, is the Lagrange polynomial of degree less than $\ladimi{1}$ which satisfies the interpolation  conditions $\pi(\varsi{j_1})=\alpha_{j_1}, ~j_1=1,\cdots,\ladimi{1}$. Clearly if all $\alpha_{j_1}=1$, the interpolating polynomial is identically equal to one.

\subsection{Decoupling of the variables: the two-variable case}

Consider now two Lagrange bases in  two variables: $\varS$ and $\varT$, defined by mutually distinct complex numbers  $\varsi{j_1}$, $j_1=1,\cdots,\ladimi{1}$ and  $\varti{j_2}$, $j_2=1,\cdots,\ladimi{2}$, respectively. We define the vectors of normalized Lagrange polynomial bases:
\begin{equation}\label{eq:basis-lagn-vect}
\begin{array}{rcl}
\bL^\vars(\varS)&=&
\left[\begin{array}{ccc} 
\lagn{\varS}{\vars}{1}&\cdots&
\lagn{\varS}{\vars}{\ladimi{1}}\end{array}\right]^\top
\in\IC^\ladimi{1}[\varS],\\
\bL^\vart(\varT)&=&
\left[\begin{array}{ccc} 
\lagn{\varT}{\vart}{1}&\cdots&
\lagn{\varT}{\vart}{\ladimi{2}}\end{array}\right]^\top
\in\IC^\ladimi{2}[\varT].
\end{array}
%\bL^\mu(\bt)=\left[\begin{array}{ccc} \ell_1(\bt)&\cdots&\ell_m(\bt)\end{array}\right]^\top \in\IC^\mu[\bt].
\end{equation}

A basis for two-variable polynomials in $\varS$ and $\varT$, is given by the entries  of $\bL^\vars(\varS)\otimes\bL^\vart(\varT)$, recovering the Kronecker structure exposed in \cite{AGPV25}. 

\begin{proposition}
A polynomial $\poly(\varS,\varT)$ of degree less than $(\ladimi{1}$, $\ladimi{2})$ in the variables $\varS$, $\varT$, respectively, can be expressed in terms of the above quantities as:
\begin{equation}
\poly(\varS,\varT)=
\left[w_{1,1}\cdots w_{1,\ladimi{2}}~\cdots~w_{\ladimi{1},1}\cdots w_{\ladimi{1},\ladimi{2}}\right]\,
\pare{\bL^{\vars}(\varS)\otimes\bL^{\vart}(\varT)}.
\end{equation}
It satisfies the interpolation conditions $\poly(\varsi{j_1},\varti{j_2})=w_{j_1,j_2}$, for $j_1=1,\cdots,\ladimi{1}$ and $j_2=1,\cdots,\ladimi{2}$. 
\end{proposition}

The variables $\varS$ and $\varT$ of this polynomial can be decoupled as follows. Let 
\begin{equation}
\begin{array}{rccll}
\bTheta^\vars(\varS)&=&\left[w_{1,\ladimi{2}},\cdots, w_{\ladimi{1},\ladimi{2}}\right]^\top&\odot\bL^\vars(\varS)&\in\IC^\ladimi{1}[\varS],\\
\bTheta_1^\vart(\varT)&=&
\frac{1}{w_{1,\ladimi{2}}}\left[w_{1,1},\cdots, w_{1,\ladimi{2}}\right]^\top&\odot\bL^\vart(\varT)
&\in\IC^\ladimi{2}[\varT]\\
&\vdots&&&\\
\bTheta_\ladimi{1}^\vart(\varT)&=&
\frac{1}{w_{\ladimi{1},\ladimi{2}}}\left[w_{\ladimi{1},1},\cdots, w_{\ladimi{1},\ladimi{2}}\right]^\top&\odot\bL^\vart(\varT) &\in\IC^\ladimi{2}[\varT].
\end{array} 
\label{eq:prop-poly2}
\end{equation}

\begin{lemma} 
Using \cref{eq:prop-poly2}, the following holds true:
\begin{equation}
\poly(\varS,\varT)=\sum_{\text{row-wise}}\left(\bTheta^\vars(\varS)\otimes\bone_\ladimi{2}\right)\odot\left(\left[\begin{array}{ccc} \bTheta_1^\vart(\varT)&\cdots&\bTheta_\ladimi{1}^\vart(\varT)\end{array}\right]^\top\otimes\bone_1\right).
\end{equation}
\end{lemma}

\begin{proof}
The two-variable polynomial $\poly(\varS,\varT)$, can be expressed as follows:
$$%\small
\renewcommand{\arraystretch}{1.5}
\left[\begin{array}{cc}
\bTheta^\vars(\varS) \otimes\bone_{\ladimi{2}} & \bTheta^\vart(\varT) \otimes\bone_{1}  \\ \hline
w_{1,\ladimi{2}} \lagn{\varS}{\vars}{1} & \frac{w_{1,1}}{w_{1,\ladimi{2}}} \lagn{\varT}{\vart}{1}\\
\vdots& \vdots\\
w_{1,\ladimi{2}} \lagn{\varS}{\vars}{1} & \frac{w_{1,\ladimi{2}}}{w_{1,\ladimi{2}}} \lagn{\varT}{\vart}{k_2}\\ \hline
\vdots&\vdots\\ \hline
w_{\ladimi{1},\ladimi{2}} \lagn{\varS}{\vars}{k_1}& \frac{w_{\ladimi{1},1}}{w_{\ladimi{1},\ladimi{2}}} \lagn{\varT}{\vart}{1}\\
\vdots &\vdots\\
w_{\ladimi{1},\ladimi{2}} \lagn{\varS}{\vars}{k_11}& \frac{w_{\ladimi{1},\ladimi{2}}}{w_{\ladimi{1},\ladimi{2}}} \lagn{\varT}{\vart}{k_2}
\end{array}\right]=
\left[\begin{array}{c}
\left(\bTheta^\vars(\varS) \otimes\bone_{\ladimi{2}}\right) \odot \left( \bTheta^\vart(\varT) \otimes\bone_{1} \right) \\\hline
w_{1,1} \lagn{\varS}{\vars}{1} \lagn{\varT}{\vart}{1} \\ 
\vdots\\
w_{1,\ladimi{2}} \lagn{\varS}{\vars}{1} \lagn{\varT}{\vart}{\ladimi{2}} \\ \hline
\vdots\\ \hline
w_{\ladimi{1},1} \lagn{\varS}{\vars}{k_1} \lagn{\varT}{\vart}{1} \\ 
\vdots\\
w_{\ladimi{1},\ladimi{2}} \lagn{\varS}{\vars}{k_1} \lagn{\varT}{\vart}{\ladimi{2}} 
\end{array}\right].
$$
Therefore, the sum of the entries of the right expression yields $\poly(\varS,\varT)$. %\sq
\end{proof}

\begin{remark}
In other words, there are $1+\ladimi{1}$ decoupled single-variable functions, namely: $\poly(\varS,\varti{k_2})$, and $\bp(\varsi{j_1},\varT)$, $j_1=1,\cdots,\ladimi{1}$. By switching the role of the variables $\varS$ and $\varT$ we obtain $1+\ladimi{2}$ decoupled single-variable functions, namely $\bp(\varsi{k_1},\varT)$ and $\bp(\varS,\varti{j_2})$, $j_2=1,\cdots,\ladimi{2}$.
\end{remark}

%\subsection{The three-variable case}
\subsection{Decoupling of the variables: the three-variable case}

To convey the general idea of variable decoupling, avoiding involved notation, we examine the special case of three-variable polynomials along $(\varS,\varT,\varZ)$, with degrees $(\sysordi{1},\sysordi{2},\sysordi{3})=(1,2,1)$, i.e., $(\ladimi{1},\ladimi{2},\ladimi{3})=(2,3,2)$, respectively. For mutually exclusive $\varsi{j_1}$, $\varti{j_2}$, $\varzi{j_3}$, ($j_1=1,2$, $j_2=1,2,3$ and $j_3=1,2$), the normalized Lagrange basis vectors are:
\begin{equation}\label{NL3}%\small
\renewcommand{\arraystretch}{1.5}
\begin{array}{rcll}
\bL^\vars(\varS)&=&\left[\lagn{\varS}{\vars}{1}~~\lagn{\varS}{\vars}{2}\right]^\top&\in\IC^{\ladimi{1}}[\varS],\\
\bL^\vart(\varT)&=&\left[\lagn{\varT}{\vart}{1}~~\lagn{\varT}{\vart}{2}~~\lagn{\varT}{\vart}{3} \right]^\top&\in\IC^{\ladimi{2}}[\varT],\\ 
\bL^\varz(\varZ)&=&\left[\lagn{\varZ}{\varz}{1}~~\lagn{\varZ}{\varz}{2}\right]^\top&\in\IC^{\ladimi{3}}[\varZ],
\end{array}
\end{equation}
where each entry is,
\begin{equation}%\label{NL3}%\small
\renewcommand{\arraystretch}{1.5}
\begin{array}{l}
\lagn{\varS}{\vars}{1}=\frac{\varS-\varsi{2}}{\varsi{1}-\varsi{2}},~
\lagn{\varS}{\vars}{2}=\frac{\varS-\varsi{1}}{\varsi{2}-\varsi{1}}\\
\lagn{\varT}{\vart}{1}=\frac{(\varT-\varti{2})(\varT-\varti{3})}{(\varti{1}-\varti{2})(\varti{1}-\varti{3})},~
\lagn{\varT}{\vart}{2}=\frac{(\varT-\varti{1})(\varT-\varti{3})}{(\varti{2}-\varti{1})(\varti{2}-\varti{3})},~
\lagn{\varT}{\vart}{3}=\frac{(\varT-\varti{1})(\varT-\varti{2})}{(\varti{3}-\varti{1})(\varti{3}-\varti{2})} \\
\lagn{\varZ}{\varz}{1}=\frac{\varZ-\varzi{2}}{\varzi{1}-\varzi{2}},~
\lagn{\varZ}{\varz}{2}=\frac{\varZ-\varzi{1}}{\varzi{2}-\varzi{1}}.
\end{array}
\end{equation}

Assuming that $\ladimi{1}\ladimi{2}\ladimi{3}$ interpolation values are given: $\poly(\varsi{j_1},\varti{j_2},\varzi{j_3})=w_{j_1,j_2,j_3}$, $j_1=1,2$, $j_2=1,2,3$, and $j_3=1,2$, define the nine coefficient vectors:
\begin{equation}
\renewcommand{\arraystretch}{1.5}
\begin{array}{l}
\pi^\vars=\left[w_{1,3,2} ,w_{2,3,2}\right]^\top\\ %\hline
\pi^\vart_1=\frac{1}{w_{1,3,2}}\left[w_{1,1,2},w_{1,2,2},w_{1,3,2}\right]^\top,~~
\pi^\vart_2=\frac{1}{w_{2,3,2}}\left[w_{2,1,2},w_{2,2,2},w_{2,3,2}\right]^\top\\ %\hline
\pi^\varz_{1,1}=\frac{1}{w_{1,1,2}}\left[w_{1,1,1},w_{1,1,2}\right]^\top,~~
\pi^\varz_{1,2}=\frac{1}{w_{1,2,2}}\left[w_{1,2,1},w_{1,2,2}\right]^\top,~~ \\
\pi^\varz_{1,3}=\frac{1}{w_{1,3,2}}\left[w_{1,3,1},w_{1,3,2}\right]^\top,~~
\pi^\varz_{2,2}=\frac{1}{w_{2,2,2}}\left[w_{2,2,1},w_{2,2,2}\right]^\top,~~ \\
\pi^\varz_{2,1}=\frac{1}{w_{2,1,2}}\left[w_{2,1,1},w_{2,1,2}\right]^\top,~~
\pi^\varz_{2,3}=\frac{1}{w_{2,3,2}}\left[w_{2,3,1},w_{2,3,2}\right]^\top.
\end{array}
\label{eq:single3var}
\end{equation}

Making use of these quantities, we define the polynomial vectors of dimension $\ladimi{1}=2$, $\ladimi{1}\ladimi{2}=6, \ladimi{1}\ladimi{2}\ladimi{3}=12$, respectively: 
\small
\begin{equation}
\renewcommand{\arraystretch}{1.5}
\begin{array}{rcl}
\bTheta^\vars(\varS)&=&\pi^\vars\odot\bL^\vars(\varS)\\
\bTheta^\vart(\varT)&=&\left[\pi^\vart_1\odot\bL^\vart(\varT),\pi^\vart_2\odot\bL^\vart(\varT)\right]^\top\\
\bTheta^\varz(\varZ)&=&\left[
\pi^\varz_{1,1}\odot\bL^\varz(\varZ) ,
\pi^\varz_{1,2}\odot\bL^\varz(\varZ) ,
\pi^\varz_{1,3}\odot\bL^\varz(\varZ), \cdots \right. \\ 
&&
\left. \pi^\varz_{2,1}\odot\bL^\varz(\varZ) ,
\pi^\varz_{2,2}\odot\bL^\varz(\varZ) ,
\pi^\varz_{2,3}\odot\bL^\varz(\varZ)\right]^\top
\end{array}
\end{equation}
\normalsize

\begin{proposition}\label{prop:decoupling_three}
The polynomial $\poly(\varS,\varT,\varZ)$ is recovered from the following nine decoupled single-variable functions $\bphi(\varS)=\poly(\varS,\varti{3},\varzi{2})$,  $\bpsi_{j_1}(\varT)=\poly(\varsi{j_1},\varT,\varzi{2})$, and $\bomega_{j_1,j_2}=\poly(\varsi{j_1},\varti{j_2},\varZ)$, $j_1=1,2$, $j_2=1,2,3$:
\begin{equation}\label{eq:kst-lf}
\poly(\varS,\varT,\varZ)=\sum_{\text{row-wise}}\left(\bTheta^\vars(\varS)\otimes\bone_{6}\right)\odot \left(\bTheta^\vart(\varT)\otimes\bone_{2}\right)\odot \left(\bTheta^\varz(\varZ)\otimes\bone_{1}\right).
\end{equation}
\end{proposition}

\begin{proof}{%\rm
The following relations hold:
$$
\footnotesize
\renewcommand{\arraystretch}{1.5}
\left[
\begin{array}{ccc}
\bTheta^\vars(\varS) \otimes\bone_6 & \bTheta^\vart(\varT) \otimes\bone_2 & \bTheta^\varz(\varZ) \otimes\bone_1 \\ \hline
\cancel{w_{1,3,2}}\cdot \lagn{\varS}{\vars}{1} & \frac{\cancel{w_{1,1,2}}}{\cancel{w_{1,3,2}}}\cdot\lagn{\varT}{\vart}{1} & \frac{w_{1,1,1}}{\cancel{w_{1,1,2}}}\cdot\lagn{\varZ}{\varz}{1} \\
\cancel{w_{1,3,2}}\cdot\lagn{\varS}{\vars}{1} & \frac{w_{1,1,2}}{\cancel{w_{1,3,2}}}\cdot\lagn{\varT}{\vart}{1} & 1\cdot\lagn{\varZ}{\varz}{2} \\
\cancel{w_{1,3,2}}\cdot\lagn{\varS}{\vars}{1} & \frac{\cancel{w_{1,2,2}}}{\cancel{w_{1,3,2}}}\cdot\lagn{\varT}{\vart}{2} & \frac{w_{1,2,1}}{\cancel{w_{1,2,2}}}\cdot\lagn{\varZ}{\varz}{1} \\
\cancel{w_{1,3,2}}\cdot\lagn{\varS}{\vars}{1} & \frac{w_{1,2,2}}{\cancel{w_{1,3,2}}}\cdot\lagn{\varT}{\vart}{2} & 1\cdot\lagn{\varZ}{\varz}{2} \\
\cancel{w_{1,3,2}}\cdot\lagn{\varS}{\vars}{1} & 1\cdot\lagn{\varT}{\vart}{3} & \frac{w_{1,3,1}}{\cancel{w_{1,3,2}}}\cdot\lagn{\varZ}{\varz}{1} \\
w_{1,3,2}\cdot\lagn{\varS}{\vars}{1} & 1\cdot\lagn{\varT}{\vart}{3} & 1\cdot\lagn{\varZ}{\varz}{2} \\ \hline
\cancel{w_{2,3,2}}\cdot\lagn{\varS}{\vars}{2} & \frac{\cancel{w_{2,1,2}}}{\cancel{w_{2,3,2}}}\cdot\lagn{\varT}{\vart}{1} & \frac{w_{2,1,1}}{\cancel{w_{2,1,2}}}\cdot\lagn{\varZ}{\varz}{1} \\
\cancel{w_{2,3,2}}\cdot\lagn{\varS}{\vars}{2} & \frac{w_{2,1,2}}{\cancel{w_{2,3,2}}}\cdot\lagn{\varT}{\vart}{1} & 1\cdot\lagn{\varZ}{\varz}{2} \\
\cancel{w_{2,3,2}}\cdot\lagn{\varS}{\vars}{2} & \frac{\cancel{w_{2,2,2}}}{\cancel{w_{2,3,2}}}\cdot\lagn{\varT}{\vart}{2} & \frac{w_{2,2,1}}{\cancel{w_{2,2,2}}}\cdot\lagn{\varZ}{\varz}{1} \\
\cancel{w_{2,3,2}}\cdot\lagn{\varS}{\vars}{2} & \frac{w_{2,2,2}}{\cancel{w_{2,3,2}}}\cdot\lagn{\varT}{\vart}{2} & 1\cdot\lagn{\varZ}{\varz}{2} \\
\cancel{w_{2,3,2}}\cdot\lagn{\varS}{\vars}{2} & 1\cdot\lagn{\varT}{\vart}{3} & \frac{w_{2,3,1}}{\cancel{w_{2,3,2}}}\cdot\lagn{\varZ}{\varz}{1} \\
w_{2,3,2}\cdot\lagn{\varS}{\vars}{2} & 1\cdot\lagn{\varT}{\vart}{3} & 1\cdot\lagn{\varZ}{\varz}{2} \\
\end{array}~\right]
=
\underbrace{\left[\begin{array}{c}
\left(\bTheta^\vars(\varS)\otimes\bone_{6}\right)\odot \left(\bTheta^\vart(\varT)\otimes\bone_{2}\right)\odot \left(\bTheta^\varz(\varZ)\otimes\bone_{1}\right)
\\ \hline
w_{1,1,1}\cdot \lagn{\varS}{\vars}{1}\,\lagn{\varT}{\vart}{1}\,\lagn{\varZ}{\varz}{1}\\
w_{1,1,2}\cdot \lagn{\varS}{\vars}{1}\,\lagn{\varT}{\vart}{1}\,\lagn{\varZ}{\varz}{2}\\\hline
w_{1,2,1}\cdot \lagn{\varS}{\vars}{1}\,\lagn{\varT}{\vart}{2}\,\lagn{\varZ}{\varz}{1}\\
w_{1,2,2}\cdot\lagn{\varS}{\vars}{1}\,\lagn{\varT}{\vart}{2}\,\lagn{\varZ}{\varz}{2}\\ \hline
w_{1,3,1}\cdot \lagn{\varS}{\vars}{1}\,\lagn{\varT}{\vart}{3}\,\lagn{\varZ}{\varz}{1}\\
w_{1,3,2}\cdot \lagn{\varS}{\vars}{1}\,\lagn{\varT}{\vart}{3}\,\lagn{\varZ}{\varz}{2}\\ \hline
w_{2,1,1}\cdot \lagn{\varS}{\vars}{2}\,\lagn{\varT}{\vart}{1}\,\lagn{\varZ}{\varz}{1}\\
w_{2,1,2}\cdot \lagn{\varS}{\vars}{2}\,\lagn{\varT}{\vart}{1}\,\lagn{\varZ}{\varz}{2}\\ \hline
w_{2,2,1}\cdot \lagn{\varS}{\vars}{2}\,\lagn{\varT}{\vart}{2}\,\lagn{\varZ}{\varz}{1}\\
w_{2,2,2}\cdot \lagn{\varS}{\vars}{2}\,\lagn{\varT}{\vart}{2}\,\lagn{\varZ}{\varz}{2}\\ \hline
w_{2,3,1}\cdot \lagn{\varS}{\vars}{2}\,\lagn{\varT}{\vart}{3}\,\lagn{\varZ}{\varz}{1}\\
w_{2,3,2}\cdot \lagn{\varS}{\vars}{2}\,\lagn{\varT}{\vart}{3}\,\lagn{\varZ}{\varz}{2}\\
\end{array}\right]}_{\br(\varS,\varT,\varZ)}\\%[-2mm]
$$
\normalsize
It follows that $\poly$ is the sum of the entries of the vector $\br$,  and consequently the interpolation conditions are satisfied, and hence $\poly(\varsi{j_1},\varti{j_2},\varzi{j_3})=\sum\br(\varsi{j_1},\varti{j_2},\varzi{j_3})=w_{j_1,j_2,j_3}$. Notice that the first column is not normalized.}
\end{proof}
Notice also that the ordering of the subscripts in $\br$, follows the ordering imposed by the Kronecker  product $\Blag{\varS}\otimes\Blag{\varT}\otimes\Blag{\varZ}$, where $\Blag{\varS}=\left[\frac{1}{\varS-\varsi{1}}~\cdots~ \frac{1}{\varS-\varsi{\ladimi{1}}}\right]^\top\,\in\,\IC^\ladimi{1}[\varS]$, $\Blag{\varT}=\left[\frac{1}{\varT-\varti{1}}~\cdots~ \frac{1}{\varT-\varti{\ladimi{2}}}\right]^\top\,\in\,\IC^\ladimi{2}[\varT]$ and $\Blag{\varZ}=\left[\frac{1}{\varZ-\varzi{1}}~\cdots~ \frac{1}{\varZ-\varzi{\ladimi{3}}}\right]^\top\,\in\,\IC^\ladimi{3}[\varZ]$.

\begin{remark}
Notice that \cref{eq:kst-lf} is a sum (over the rows) of products (over the vectors). In order to recover the KST sum of sum formula \cref{eq:kst-original}, one should simply involve the outer $\exp(.)$ of the sum of inner $\log(.)$ of each vector function instead, thus leading to 
\begin{equation}\label{eq:kst_style}
\poly(\varS,\varT,\varZ)=\sum_{\text{row-wise}}\exp\left( \left(\log(\bTheta^\vars(\varS))\otimes\bone_{6}\right)+\left(\log(\bTheta^\vart(\varT))\otimes\bone_{2}\right)+\left(\log(\bTheta^\varz(\varZ))\otimes\bone_{1}\right)\right).
\end{equation}
\end{remark}

\begin{remark}
\Cref{prop:decoupling_three} reports nine decoupled single-variable functions, namely $\bphi(\varS)=\poly(\varS,\varti{3},\varzi{2})$,  $\bpsi_{j_1}(\varT)=\poly(\varsi{j_1},\varT,\varzi{2})$ and $\bomega_{j_1,j_2}=\poly(\varsi{j_1},\varti{j_2},\varZ)$, $j_1=1,2$, $j_2=1,2,3$. These functions are related to the data given in  \cref{eq:single3var}, where 
\begin{equation}
\renewcommand{\arraystretch}{1.5}
\begin{array}{l}
\pi^\vars=\bphi(\varsi{1,2}) \\
\pi^\vart_1=\frac{1}{\bpsi_1(\varti{3})}\bpsi_1(\varti{1,2,3}) ~,~\pi^\vart_2=\frac{1}{\bpsi_2(\varti{3})}\bpsi_2(\varti{1,2,3}) \\
\pi^\varz_{1,1}=\frac{1}{\bomega_{1,1}(\varzi{2})}\bomega_{1,1}(\varzi{1,2})~,~
\pi^\varz_{1,2}=\frac{1}{\bomega_{1,2}(\varzi{2})}\bomega_{1,2}(\varzi{1,2})~,~\\
\pi^\varz_{1,3}=\frac{1}{\bomega_{1,3}(\varzi{2})}\bomega_{1,3}(\varzi{1,2})~,~ 
\pi^\varz_{2,1}=\frac{1}{\bomega_{2,1}(\varzi{2})}\bomega_{2,1}(\varzi{1,2})~,~\\
\pi^\varz_{2,2}=\frac{1}{\bomega_{2,2}(\varzi{2})}\bomega_{2,2}(\varzi{1,2})~,~
\pi^\varz_{2,3}=\frac{1}{\bomega_{2,3}(\varzi{2})}\bomega_{2,3}(\varzi{1,2}).
\end{array}
\end{equation}
\end{remark}

\subsection{Recursive computation of barycentric weights: the four-variable case}

\subsubsection{General approach} 

Similarly to \cite{AGPV25}, it is now possible to construct the interpolant recursively. Given is a polynomial in four variables: $\varS$, $\varT$, $\varX$, $\varZ$, of degree $\ladimi{i}$, $i=1,\cdots,4$, respectively. The normalized (except for the first variable) Lagrange bases in each variable are spanned by $\Blag{\varS}$, $\Blag{\varT}$, $\Blag{\varX}$ and $\Blag{\varZ}$. %the monomials $\varS-\varsi{j_1}$, $\varT-\varti{j_2}$, $\varX-\varxi{j_3}$ and $\varZ-\varzi{j_4}$, $(j_1,j_2,j_3,j_4)=1,\cdots,(\ladimi{1},\ladimi{2},\ladimi{3},\ladimi{4})$. 
Similarly to \cref{eq:basis-lagn-vect}, let us denote the associated vectors by $\bL^\vars(\varS)\in\IC^\ladimi{1}[\varS]$, $\bL^\vart(\varT)\in\IC^\ladimi{2}[\varT]$, $\bL^\varx(\varX)\in\IC^\ladimi{3}[\varX]$, $\bL^\varz(\varZ)\in\IC^\ladimi{4}[\varZ]$.  Define the normalized barycentric weights for each variable (notice that each quantity can be independently computed, i.e., in parallel):
$$
\renewcommand{\arraystretch}{1.5}
\overbrace{\left[
\begin{array}{c}
\poly(\varsi{1},\varT,\varX,\varZ) \\
\vdots\\
\poly(\varsi{\ladimi{1}-1},\varT,\varX,\varZ) \\
\poly(\varsi{\ladimi{1}},\varT,\varX,\varZ)
\end{array}\right]}^{\displaystyle\bq_\vars(\varT,\varX,\varZ)}~~
\overbrace{\left[
\begin{array}{c}
\frac{\poly(\varS,\varti{1},\varT,\varZ)}{\poly(\varS,\varti{\ladimi{2}},\varT,\varZ)} \\
\vdots\\
\frac{\poly(\varS,\varti{\ladimi{2}-1},\varT,\varZ)}{\poly(\varS,\varti{\ladimi{2}},\varT,\varZ)}\\
{1}\end{array}\right]}^{\displaystyle\bq_\vart(\varS,\varX,\varZ)}
~~
\overbrace{\left[
\begin{array}{c}
\frac{\poly(\varS,\varT,\varxi{1},\varZ)}{\poly(\varS,\varT,\varxi{\ladimi{3}},\varZ)}\\
\vdots\\
\frac{\poly(\varS,\varT,\varxi{\ladimi{3}-1},\varZ)}{\poly(\varS,\varT,\varxi{\ladimi{3}},\varZ)}\\
{1}\end{array}\right]}^{\displaystyle\bq_\varx(\varS,\varT,\varZ)}~~
\overbrace{\left[
\begin{array}{c}
\frac{\poly(\varS,\varT,\varX,\varzi{1})}{\poly(\varS,\varT,\varX,\varzi{\ladimi{4}})}\\
\vdots\\
\frac{\poly(\varS,\varT,\varX,\varzi{\ladimi{4}-1})}{\poly(\varS,\varT,\varX,\varzi{\ladimi{4}})}\\
{1}\end{array}\right].}^{\displaystyle\bq_\varz(\varS,\varT,\varX)}
$$
From the above data the quantities $\bZ(\varS,\varT,\varX)$, $\bX(\varS,\varT)$, $\bT(\varS)$, $\bS$ are computed recursively, and it holds that: $\displaystyle \sum_{\text{row-wise}} \bS=\poly(\varS,\varT,\varX,\varZ)$, where
$$%\small
\renewcommand{\arraystretch}{1.5}
\begin{array}{rcll}
\bZ(\varS,\varT,\varX)&:=&\left(\bq_\varz(\varS,\varT,\varX)\odot\bL^\varz(\varZ)\right)\otimes\bone_1
&\in\IC^{\ladimi{4}},\\
\bX(\varS,\varT)&:=&\left[\begin{array}{c}
\bZ(\varS,\varT,\varxi{1})\\
\vdots\\
\bZ(\varS,\varT,\varxi{\ladimi{3}})
\end{array}\right]\odot (\bq_\varx(\varS,\varT,z_{\nu_4})\odot\bL^\varx(\varX))\otimes\bone_{\ladimi{3}}
&\in\IC^{\ladimi{3}\ladimi{4}}\\
\bT(\varS)&:=&\left[\begin{array}{c}
\bX(\varS,\varti{1})\\
\vdots\\
\bX(\varS,\varti{\ladimi{2}})
\end{array}\right]\odot (\bq_\vart(\varS,\varxi{\ladimi{3}},\varzi{\ladimi{4}})\odot\bL^\vart(\varT))\otimes\bone_{\ladimi{3}\ladimi{4}}
&\in\IC^{\ladimi{2}\ladimi{3}\ladimi{4}}, \\
\bS&:=&\left[\begin{array}{c}
\bT(\varsi{1})\\
\vdots\\
\bT(\varsi{\ladimi{1}})
\end{array}\right]\odot (\bq_\vars(\varti{\ladimi{2}},\varxi{\ladimi{3}},\varzi{\ladimi{4}})\odot\bL^\bs)\otimes\bone_{\ladimi{2}\ladimi{3}\ladimi{4}}
&\in\IC^{\ladimi{1}\ladimi{2}\ladimi{3}\ladimi{4}}.
\end{array}
$$

\subsubsection{Example}
\label{app:rec-poly}

Given is the polynomial $\poly(\varS,\varT,\varX,\varZ) = \varX^2+\varS \varX \varZ+\varT \varZ^2+1$, and thus $\ladimi{1}=\ladimi{2}=2$, $\ladimi{3}=\ladimi{4}=3$. Following \eqref{eq:basis-lagn-vect}, the normalized Lagrange bases for each variable are: 
\begin{equation}\label{LN4}
\renewcommand{\arraystretch}{1.5}
\begin{array}{c}
\bL^\vars(\varS) =
\left[\begin{array}{c} \frac{\varS-\varsi{2}}{\varsi{1}-\varsi{2}}\\[1mm] \frac{\varS-\varsi{1}}{\varsi{2}-\varsi{1}} \end{array}\right],~
\bL^\vart(\varT) =
\left[\begin{array}{c} \frac{\varT-\varti{2}}{\varti{1}-\varti{2}}\\[1mm] \frac{\varT-\varti{1}}{\varti{2}-\varti{1}} \end{array}\right],~\\[1mm] 
\bL^\varx(\varX) =
\left[\begin{array}{c} 
\frac{\left(\varX-\varxi{2}\right) \left(\varX-\varxi{3}\right)}{\left(\varxi{1}-\varxi{2}\right) \left(\varxi{1}-\varxi{3}\right)}\\[1mm] 
\frac{\left(\varX-\varxi{1}\right) \left(\varX-\varxi{3}\right)}{\left(\varxi{2}-\varxi{1}\right) \left(\varxi{2}-\varxi{3}\right)}\\[1mm] \frac{\left(\varX-\varxi{1}\right) \left(\varX-\varxi{2}\right)}{\left(\varxi{1}-\varxi{3}\right) \left(\varxi{2}-\varxi{3}\right)} \end{array}\right],~
\bL^\varz(\varZ) =
\left[\begin{array}{c} \frac{\left(\varZ-\varzi{2}\right) \left(\varZ-\varzi{3}\right)}{\left(\varzi{2}-\varzi{2}\right) \left(\varzi{2}-\varzi{3}\right)}\\[1mm] \frac{\left(\varZ-\varzi{1}\right) \left(\varZ-\varzi{3}\right)}{\left(\varzi{2}-\varzi{1}\right) \left(\varzi{2}-\varzi{3}\right)}\\[1mm] \frac{\left(\varZ-\varzi{1}\right) \left(\varZ-\varzi{2}\right)}{\left(\varzi{1}-\varzi{3}\right) \left(\varzi{2}-\varzi{3}\right)} \end{array}\right].
\end{array}
\end{equation}
Then, the barycentric weights are (notice the first variable is not normalized):
\begin{equation}
\begin{array}{rclcl}
\bq_\vars(\varT,\varX,\varZ)&=&
\left[\begin{array}{c}
\poly(\varsi{1} ,\varT,\varX,\varZ)\\
\poly(\varsi{2} ,\varT,\varX,\varZ)
\end{array}\right] &=& 
\left[\begin{array}{c} \varX^2+\varsi{1} \varX \varZ+\varT \varZ^2+1\\
\varX^2+\varsi{2} \varX \varZ+\varT \varZ^2+1\end{array}\right],~ \\
\bq_\vart(\varS,\varX,\varZ)&=&
\left[\begin{array}{c} 
\frac{\poly(\varS,\varti{1},\varX,\varZ)}{\poly(\varS,\varti{2},\varX,\varZ)}\\
1 \end{array}\right]&=&
\left[\begin{array}{c} 
\frac{\varX^2+\varS \varX \varZ+\varti{1} \varZ^2+1}{\varX^2+\varS \varX \varZ+\varti{2} \varZ^2+1}\\
1 \end{array}\right], \\ 
\bq_\varx(\varS,\varT,\varZ)&=&
\left[\begin{array}{c} 
\frac{\poly(\varS,\varT,\varxi{1},\varZ)}{\poly(\varS,\varT,\varxi{3},\varZ)}\\
\frac{\poly(\varS,\varT,\varxi{2},\varZ)}{\poly(\varS,\varT,\varxi{3},\varZ)}\\
1 \end{array}\right] &=&
\left[\begin{array}{c} 
\frac{\varxi{1}^2+\varS \varxi{1} \varZ+\varT \varZ^2+1}{\varxi{3}^2+\varS \varxi{3} \varZ+\varT \varZ^2+1}\\
\frac{\varxi{2}^2+\varS \varxi{2} \varZ+\varT \varZ^2+1}{\varxi{3}^2+\varS \varxi{3} \varZ+\varT \varZ^2+1}\\
1 \end{array}\right], \\ 
\bq_\varz(\varS,\varT,\varX)&=&
\left[\begin{array}{c} 
\frac{\poly(\varS,\varT,\varX,\varzi{1})}{\poly(\varS,\varT,\varX,\varzi{3})}\\
\frac{\poly(\varS,\varT,\varX,\varzi{2})}{\poly(\varS,\varT,\varX,\varzi{3})}\\
1 \end{array}\right]&=&
\left[\begin{array}{c} 
\frac{\varX^2+\varS \varX \varzi{1}+\varT \varzi{1}^2+1}{\varX^2+\varS \varX \varzi{3}+\varT \varzi{3}^2+1}\\
\frac{\varX^2+\varS \varX \varzi{2}+\varT \varzi{2}^2+1}{\varX^2+\varS \varX \varzi{3}+\varT \varzi{3}^2+1}\\
1 \end{array}\right].
\end{array}
\end{equation}
Then, the procedure is as follows:
\footnotesize
$$
\begin{array}{lrcccll}
\mbox{step}~1:&\bZ&=& & & \left(\bq_\varz(\varS,\varT,\varX) \odot \bL_\varz(\varZ) \right) & \in\IC^{\ladimi{4}}[\varS,\varT,\varX,\varZ]\\[1mm]
\mbox{step}~2:&\bX&=&\left[\begin{array}{l}
\bZ\mid_{\varX=\varxi{1}}\\[1mm] 
\bZ\mid_{\varX=\varxi{2}}\\[1mm] 
\bZ\mid_{\varX=\varxi{3}}
\end{array}\right] & \odot & \left(\bq_\varx(\varS,\varT,\varzi{3}) \odot \bL_\varx(\varX)\right)\otimes\bI_{\ladimi{4}} & \in\IC^{\ladimi{3}\ladimi{4}}[\varS,\varT,\varX,\varZ]\\[1mm]
\mbox{step}~3:&\bT&=&\left[\begin{array}{l}
\bX\mid_{\bt=\varti{1}}\\[1mm] \bX\mid_{\bt=\varti{2}}\end{array}\right] & \odot & \left( \bq_\vart(\varS,\varxi{3},\varzi{3}) \odot \bL_{\vart}(\varT)\right)\otimes \bI_{\ladimi{3}\ladimi{4}}
&\in\IC^{\ladimi{2}\ladimi{3}\ladimi{4}}[\varS,\varT,\varX,\varZ]\\[1mm]
\mbox{step}~4:&\bS&=&\left[ \begin{array}{l}
\bT\mid_{\bs=s_1}\\[1mm] 
\bT\mid_{\bs=s_2}
\end{array}\right] & \odot & \left(\bq_\vars(\varti{2},  \varxi{3},\varzi{3}) \odot  \bL_{\vars}(\varS)\right) \otimes \bI_{\ladimi{2}\ladimi{3}\ladimi{4}} & \in\IC^{\ladimi{1}\ladimi{2}\ladimi{3}\ladimi{4}}[\varS,\varT,\varX,\varZ]
\end{array}
$$
\normalsize
This leads to:
$$
\left\{\begin{array}{lll}
\sum_{\text{row-wise}} \bZ=\frac{\num^\varz}{\den^\varz}& 
\num^z =\varX^2 + \varS \varX \varZ + \varT \varZ^2 + 1,& 
\den^z =\varX^2 + \varS \varX \varzi{3} + \varT \varzi{3}^2 + 1,\\[1mm]
\sum_{\text{row-wise}} \bX=\frac{\num^x}{\den^x}& 
\num^x =\varX^2 + \varS \varX \varZ + \varT \varZ^2 + 1,& 
\den^x =\varxi{3}^2 + \varS \varxi{3} \varzi{3} + \varT \varzi{3}^2 + 1,\\[1mm]
\sum_{\text{row-wise}} \bT=\frac{\num^t}{\den^t}& 
\num^t =\varX^2 + \varS \varX \varZ + \varT \varZ^2 + 1,& 
\den^t =\varxi{3}^2 + \varS \varxi{3} \varzi{3} + \varti{2} \varzi{3}^2 + 1,\\[1mm]
\sum_{\text{row-wise}} \bS=\frac{\num^s}{\den^s}& 
\num^s =\varX^2 + \varS \varX \varZ + \varT \varZ^2 + 1,& 
\den^s =1
\end{array}\right.
$$
Thus, summing over the last step leads to the result, i.e., $\sum_{\text{row-wise}}  \bS=\num^\vars=\bp(\varS,\varT,\varX,\varZ)$. In \ref{app:rec-rat}, a detailed example of rational approximation using the recursive procedure is given.

\section{The P\'olya-Szeg\"o three-variable example}\label{sec:polya-szego}

\subsection{Variable decoupling}

A well-known solution for the  P\'olya-Szeg\"o  variable decoupling is given by\footnote{The relation in \eqref{eq:polya:gemini} have been provided by the Gemini software. Notice also that if, in the  P\'olya-Szeg\"o  polynomial, the multiplying coefficients before of each summed elements are not equals, the solution involves 7 summed elements instead of 4.}
\begin{equation}
\poly(\varX,\varY,\varZ)=\sum_{k=1}^4 \bg_k (\bphi_{k,1}(\varX)+\bphi_{k,2}(\varY)+ \bphi_{k,3}(\varZ)),
\end{equation}
where,
\begin{equation}
\begin{array}{r|l|l}
k=1 & \bphi_{1,1}(\varX) =\varX,~\bphi_{1,2}(\varY)=\varY,~\bphi_{1,3}(\varZ)=\varZ & \bg_1(w)=\frac{1}{2}w^2\\
k=2 & \bphi_{2,1}(\varX) =\varX,~\bphi_{2,2}(\varY)=0,~\bphi_{2,3}(\varZ)=0 & \bg_2(w)=\frac{-1}{2}w^2\\
k=3 & \bphi_{3,1}(\varX) =0,~\bphi_{3,2}(\varY)=\varY,~\bphi_{3,3}(\varZ)=0 & \bg_3(w)=\frac{-1}{2}w^2\\
k=4 & \bphi_{4,1}(\varX) =0,~\bphi_{2,2}(\varY)=0,~\bphi_{2,3}(\varZ)=\varZ & \bg_4(w)=\frac{-1}{2}w^2.
\end{array}
\label{eq:polya:gemini}
\end{equation}

\subsection{Interpolation without computation}\label{ssec:inspection}

Given the three-variable polynomial studied in \cite{polya}: $\poly({\varX},{\varY},{\varZ})=\varX \varY+\varX\varZ+\varY\varZ$, where $\ladimi{1}=\ladimi{2}=\ladimi{3}=2$, $\ladimi{1}\ladimi{2}\ladimi{3}=8$. This is discussed in problem \#119 of \cite{polya}. Let the interpolation points be: $(\varxi{1},\varxi{2})$, $(\varyi{1},\varyi{2})$, $(\varzi{1},\varzi{2})$, respectively. The $1+\ladimi{1}+\ladimi{1}\ladimi{2}=7$, decoupled, single-variable  functions are:
$\phi(\varX)=\poly(\varX,\varyi{2},\varzi{2})$, 
$\psi_1(\varY)=\poly(\varxi{1},\varY,\varzi{2})$, 
$\psi_2(\varY)=\poly(\varxi{2},\varY,\varzi{2})$, 
$\omega_1(\varZ)=\poly(\varxi{1},\varyi{1},\varZ)$, 
$\omega_2(\varZ)=\poly(\varxi{2},\varyi{1},\varZ)$, 
$\omega_3(\varZ)=\poly(\varxi{1},\varyi{2},\varZ)$, 
$\omega_4(\varZ)=\poly(\varxi{2},\varyi{2},\varZ)$. 

Putting these together as described above yields the following three vector-valued, single-variable functions. The decoupled polynomial vector functions are:
$$
\renewcommand{\arraystretch}{1.5}
\begin{array}{rcll}
\bPhi({\bx})&=&
\left[\begin{array}{cc} 
\frac{\varX-\varxi{2}}{\varxi{1}-\varxi{2}} 
\phi(\varxi{1}) & 
\frac{\varX-\varxi{1}}{\varxi{2}-\varxi{1}} \phi(\varxi{2}) \\ %[1mm]
\end{array}\right]^\top&\in \IC^2[\bx], \\
\bPsi(\varY) &=&
\left[\begin{array}{cc|cc} 
\frac{\varY-\varyi{2}}{\varyi{1}-\varyi{2}} 
\frac{\psi_1(\varyi{1})}{\psi_1(\varyi{2})}&
\frac{{\varY}-\varyi{1}}{\varyi{2}-\varyi{1}}& 
\frac{{\varY}-\varyi{2}}{\varyi{1}-\varyi{2}} \frac{\psi_2(\varyi{1})}{\psi_2(\varyi{2})} &
\frac{\varY-\varyi{1}}{\varyi{2}-\varyi{1}} \\ 
\end{array}\right]^\top&\in \IC^4[{\by}], \\ %[1mm]
\bOmega(\varZ) &=&
\left[\begin{array}{cc|cc}%|cc|cc} 
\frac{\varZ-\varzi{2}}{\varzi{1}-\varzi{2}} 
\frac{\omega_1(\varzi{1})}{\omega_1(\varzi{2})} &
\frac{\varZ-\varzi{1}}{\varzi{2}-\varzi{1}} &
\frac{\varZ-\varzi{2}}{\varzi{1}-\varzi{2}} 
\frac{\omega_3(\varzi{1})}{\omega_3(\varzi{2})} &
\frac{\varZ-\varzi{1}}{\varzi{2}-\varzi{1}}  \cdots 
\end{array} \right. \\
&&
\left.
\begin{array}{cc|cc}
\frac{\varZ-\varzi{2}}{\varzi{1}-\varzi{2}} 
\frac{\omega_2(\varzi{1})}{\omega_2(\varzi{2})} &
\frac{\varZ-\varzi{1}}{\varzi{2}-\varzi{1}} &
\frac{\varZ-\varzi{2}}{\varzi{1}-\varzi{2}} 
\frac{\omega_4(\varzi{1})}{\omega_4(\varzi{2})} &
\frac{\varZ-\varzi{1}}{\varzi{2}-\varzi{1}} 
\end{array}\right]^\top&\in\IC^8[\varZ],
\end{array}
$$
and $\widehat\bPhi(\varX)=\bPhi(\varX)\otimes\bone_4$, $\widehat{\bPsi}(\varY) =\bPsi(\varY) \otimes \bone_2$ and $\widehat{\bOmega}({\bz})=\bOmega({\bz})\otimes\bone_1$. Thus, the original function is decoupled by means of these vector functions: 
\begin{equation}
\begin{array}{rcl}
\poly(\varX,\varY,\varZ)&=& \sum_{k=1}^8 \widehat\bPhi_k(\varX) \odot \widehat\bPsi_k(\varY) \odot \widehat\bOmega_k(\varZ)\\
&=&  \sum_{k=1}^8 \exp\left(\log\left(\widehat\bPhi_k(\varX)\right) + \log\left(\widehat\bPsi_k(\varY)\right) + \log\left(\widehat\bOmega_k(\varZ)\right)\right)\\
&=& \sum_{k=1}^8 \bg_k (\bphi_{k,1}(\varX)+\bphi_{k,2}(\varY)+ \bphi_{k,3}(\varZ))\\
&=& \sum_{k=1}^8 \bg (\bphi_{k,1}(\varX)+\bphi_{k,2}(\varY)+ \bphi_{k,3}(\varZ)),
\end{array}
\end{equation}
where, for $k=1,\cdots,8$, $\bg(w)=\bg_k(w)=\exp(w)$ and
\begin{equation}
\begin{array}{r|l}
k=1 & \bphi_{1,1}(\varX) =\log(\bPhi(\varX))_1,~\bphi_{1,2}(\varY)=\log(\bPsi(\varY))_1,~\bphi_{1,3}(\varZ)=\log(\bOmega(\varZ))_1 \\
k=2 & \bphi_{2,1}(\varX) =\log(\bPhi(\varX))_1,~\bphi_{2,2}(\varY)=\log(\bPsi(\varY))_1,~\bphi_{2,3}(\varZ)=\log(\bOmega(\varZ))_2 \\
k=3 & \bphi_{3,1}(\varX) =\log(\bPhi(\varX))_1,~\bphi_{3,2}(\varY)=\log(\bPsi(\varY))_2,~\bphi_{3,3}(\varZ)=\log(\bOmega(\varZ))_3 \\
k=4 & \bphi_{4,1}(\varX) =\log(\bPhi(\varX))_1,~\bphi_{4,2}(\varY)=\log(\bPsi(\varY))_2,~\bphi_{4,3}(\varZ)=\log(\bOmega(\varZ))_4 \\
k=5 & \bphi_{5,1}(\varX) =\log(\bPhi(\varX))_2,~\bphi_{5,2}(\varY)=\log(\bPsi(\varY))_3,~\bphi_{5,3}(\varZ)=\log(\bOmega(\varZ))_5 \\
k=6 & \bphi_{6,1}(\varX) =\log(\bPhi(\varX))_2,~\bphi_{6,2}(\varY)=\log(\bPsi(\varY))_3,~\bphi_{6,3}(\varZ)=\log(\bOmega(\varZ))_6 \\
k=7 & \bphi_{7,1}(\varX) =\log(\bPhi(\varX))_2,~\bphi_{7,2}(\varY)=\log(\bPsi(\varY))_4,~\bphi_{7,3}(\varZ)=\log(\bOmega(\varZ))_7 \\
k=8 & \bphi_{8,1}(\varX) =\log(\bPhi(\varX))_2,~\bphi_{8,2}(\varY)=\log(\bPsi(\varY))_4,~\bphi_{8,3}(\varZ)=\log(\bOmega(\varZ))_8,
\end{array}
\end{equation}
where $\log(\bPhi(\varX))_k$ is the $k$-th entry of the vector function.

\begin{remark}
In this specific example, notice that each vector function is composed of simple functions of the same structure, of the type $f(x)= ax + b$, where $a,b\in\IC$. Indeed, by considering the first element of $\bPhi(\varX)$, we get $f(\varX)=\frac{\varX-\varxi{2}}{\varxi{1}-\varxi{2}} \phi(\varxi{1})$ leads to $a=\frac{\phi(\varxi{1})}{\varxi{1}-\varxi{2}}$ and $b=\frac{-\varxi{2}}{\varxi{1}-\varxi{2}}$. More generally, for a polynomial of degree $\sysord$, $f$ is also of degree $\sysord$.
\end{remark}

%\begin{remark}
%In \ref{app:polya}, this example is treated from two different perspectives: the Kolmogorov Arnold Network (KAN) and the multivariate Loewner Framework (mLF) angles. It shows how much the mLF is able to exactly recover the function with few data and a limited computational burden. This paper suggests that when the function is available, similar results can be obtained without any computation.
%\end{remark}

\subsection{Interpolation with computation}

For the sake of completeness, let us continue this example by using two numerical code implementations aiming at  discovering a multivariate P\'olya-Szeg\"o function based on data only. In  \cref{ssec:pykan}, we present a Python code using the \texttt{pyKAN} package by \cite{kan}\footnote{Available at: \url{https://github.com/KindXiaoming/pykan}}, and in Section \cref{ssec:mlf}, we present a MATLAB code using the \texttt{mLF} by \cite{AGPV25}\footnote{Available at: \url{https://github.com/cpoussot/mLF}}.

%We continue the example presented in the introduction and  detailed in \cref{sec:polya-szego}. Here, we discuss two numerical code implementations aiming at  discovering a multivariate P\'olya-Szeg\"o function based on data only. In  \cref{ssec:pykan}, we present a Python code using the \texttt{pyKAN} package by \cite{kan}\footnote{Available at: \url{https://github.com/KindXiaoming/pykan}}, and in Section \cref{ssec:mlf}, we present a Matlab code using the \texttt{mLF} by \cite{AGPV25}\footnote{Available at: \url{https://github.com/cpoussot/mLF}}.

\subsubsection{The \texttt{pyKAN} approach}\label{ssec:pykan}

We first use the \texttt{pyKAN} Python package by \cite{kan}. The code is given below.

%\begin{lstlisting}[language=Python,caption=Python code for the P\'olya-Szeg\"o example using \texttt{pyKAN}]
%\vspace{-2mm}
{\small
\begin{verbatim}
### Use pyKAN (https://kindxiaoming.github.io/pykan/index.html)
from kan import *
import numpy as np
### Function (Polya-Szego)
H = lambda x: x[:,0]*x[:,1] + x[:,0]*x[:,2] + x[:,1]*x[:,2]
### Create dataset 
dataset = create_dataset(H, n_var=3, train_num=1000)
### Create a KAN: 
# width of KAN (3 inputs, 7 hidden neurons, 1 output) 
# grid intervals (grid=5)
# cubic spline (k=3)
model = KAN(width=[3,7,1], grid=5, k=3, seed=0)
### Plot KAN at initialization
model(dataset['train_input']);
model.plot(beta=100)
### Train
model.fit(dataset, opt="LBFGS", steps=1000);
model = model.prune()
model.plot()
print(model)
### Symbolic
lib = ['x','x^2','x^3']
model.auto_symbolic(lib=lib)
model_sym = model.symbolic_formula()[0][0]
print(model_sym)
\end{verbatim}
}
%\end{lstlisting}

The block  first creates the function and the data set. Here we chose 1000 sample points for the $\ord=3$ variables problem. Then the KAN architecture is chosen, with $2\ord+1$ hidden neurons, being the one indicated by the Kolmogorov result. Then we initialize, train, prune, and plot the network. Here we chose the \texttt{LBFGS} optimization method, which led to the best results. Finally, the symbolic form of the surrogate model is computed, leading to the  following symbolic formula (rounded at third digit): $\texttt{model\_sym}(x_1,x_2,x_3) = 0.085 x_1 + 0.083 x_2 - 0.248 x_3 - 0.813 (0.700 x_1 - 0.120 x_3 - 0.030)^2 - 0.409 (-0.985 x_2 + 0.180 x_3 + 0.040)^2 - 0.466 (0.206 x_1 + 0.195 x_2 - 0.926 x_3 + 0.410)^2 + 0.040 (0.800 x_1 + 0.738 x_2 + 0.686 x_3 - 0.506)^2 + 0.618 (0.810 x_1 + 0.814 x_2 + 0.615 x_3 + 0.0117)^2 + 0.122$.
%$\texttt{model\_sym}(x_1,x_2,x_3) = 0.0850653022913471 x_1 + 0.0825853690442805 x_2 - 0.248315405767798 x_3 - 0.812635481357574 (0.69964097700502 x_1 - 0.119794319364831 x_3 - 0.0301766892079272)^2 - 0.409290462732315 \\ (-0.985099053714597 x_2 + 0.1801977654841 x_3 + 0.0402133109601216)^2 - 0.465758383274078 (0.205582950890196 x_1 + 0.195172294630299 x_2 - 0.926481644338002 x_3 + 0.410213824838751)^2 + 0.0399022586643696 (0.79880678544853 x_1 + 0.737652184852142 x_2 + 0.686408054766781 x_3 - 0.505935898189065)^2 + 0.618472099304199 (0.810037995234768 x_1 + 0.813577309716954 x_2 + 0.614917247924268 x_3 + 0.0117037574483545)^2 + 0.121694707448264$

Here, the recovered function is close to the original one with a absolute error close about $3.9\cdot10^{-2}$, however, the original function is not recovered. The next part shows how multivariate rational approximation in the Loewner framework solves the issue.

%\begin{figure}[h]
%    \centering
%    \includegraphics[width=0.7\linewidth]{figures/pyKAN_polya_szego.png}
%    \caption{KAN model for the P\'olya-Szeg\"o function obtained with \texttt{pyKAN}.}
%    \label{fig:pykan_polya_szego}
%\end{figure}

\subsubsection{The \texttt{mLF} approach}\label{ssec:mlf}

Next, the \texttt{mLF} MATLAB package by \cite{AGPV25} is used to solve the same problem. The code is given below.

%\begin{lstlisting}[language=Matlab,caption=Matlab code for the P\'olya-Szeg\"o example using \texttt{mLF}]
{\small
\begin{verbatim}
%%% Function (Polya-Szego)
H   = @(x)x(:,1).*x(:,2)+x(:,1).*x(:,3)+x(:,2).*x(:,3);
%%% Interpolation points & tensor construction
p_c             = {1:3 1:3 1:3};
p_r             = {-2:0 -2:0 -2:0};
[y,x,dim]       = mlf.make_tab_vec(H,p_c,p_r);
T               = mlf.vec2mat(y,dim); 
%%% mLF [A/G/P-V, 2025, alg. 1]
opt.ord_tol     = 1e-8;
opt.method_null = 'svd0';
opt.method      = 'rec';
[r,iloe]        = mlf.alg1(tab,p_c,p_r,opt);
%%% KST symbolic view
[Bary,Lag,~]    = mlf.decoupling(iloe.pc,iloe.lag);
num             = simplify(sum(iloe.w.*Bary{1}.*Lag{1}.* ...
                           Bary{2}.*Lag{2}.*Bary{3}.*Lag{3}));
den             = simplify(sum(Bary{1}.*Lag{1}.*Bary{2}.* ...
                           Lag{2}.*Bary{3}.*Lag{3}));
num = subs(subs(subs(simplify(num),'s1','x'),'s2','y'),'s3','z');
den = subs(subs(subs(simplify(den),'s1','x'),'s2','y'),'s3','z');
simplify(num/den)
\end{verbatim}
}
%\end{lstlisting}

The above code first creates the function and the data set. Here we chose $8$ sample points only for the $\ord=3$ variables problem (this is discussed next). Then the multivariate Lagrangian model is computed. Finally, the symbolic form of the surrogate model is computed, leading to the  following symbolic formula: 
$$
\left\{\begin{array}{rcl}
\texttt{num}(x,y,z)&=& \dfrac{8\,\left(x\,y+x\,z+y\,z\right)}{\left(x^2-4\,x+3\right)\,\left(y^2-4\,y+3\right)\,\left(z^2-4\,z+3\right)} \\
\texttt{den}(x,y,z)&=& \dfrac{8}{\left(x^2-4\,x+3\right)\,\left(y^2-4\,y+3\right)\,\left(z^2-4\,z+3\right)}
\end{array} \right. ,
$$
then we obtain ${\texttt{num}(x,y,z)}/{\texttt{den}(x,y,z)} = x\,y+x\,z+y\,z$, which exactly recovers the function. More details on the different steps and Lagrangian graph are available in \cite[\S 5.48]{PVGVA}.

\begin{remark}
    We would like to stress that the exact function is recovered. Unlike the previous case, here, only 8 points are needed to recover the function. We also note that choosing more points does not affect the result. Notice also that the complexity of the rational (and polynomial) form is automatically detected by the algorithm (see \cite{AGPV25,PVGVA} for details).    
\end{remark}

\begin{remark}
The results shown in this paper illustrate that this very same result can be obtained by inspection only, avoiding tedious and costly computational steps, as shown in \cref{ssec:inspection}.
\end{remark}

\section{Variable decoupling: the general case}\label{sec:general-case}

\subsection{Notations introduction}

For the three variables rational case ($\ord=3$), introducing the general notations to the variables $\var{1}, \var{2}, \var{3}$ (instead of $\varS$, $\varT$ and $\varZ$), we have ($l=1,\cdots,3$):
$$\small
\left\{\begin{array}{rrcl}
\bPhi_3(\var{3})=&\baryden{\var{3}} &\odot&\bone_{\ladimi{1}}\otimes \bone_{\ladimi{2}}\otimes \bX_3\\
\bPhi_2(\var{2})=& \baryden{\var{2}}\otimes\bone_{\ladimi{3}}&\odot&\bone_{\ladimi{1}}\otimes\bX_2\otimes\bone_{\ladimi{3}}\\
\bPhi_1(\var{1})=& \baryden{\var{1}}\otimes\bone_{\ladimi{2}}\otimes\bone_{\ladimi{3}}&\odot& \bX_1\otimes\bone_{\ladimi{2}}\otimes\bone_{\ladimi{3}}\\
\end{array}\right.
$$
$$
\left\{\begin{array}{rrcl}
\widehat\bPhi_3(\var{3})=&\barynum{\var{3}} &\odot&\bone_{\ladimi{1}}\otimes \bone_{\ladimi{2}}\otimes \bX_3\\
\widehat\bPhi_2(\var{2})=& \barynum{\var{2}}\otimes\bone_{\ladimi{3}}&\odot&\bone_{\ladimi{1}}\otimes\bX_2\otimes\bone_{\ladimi{3}}\\
\widehat\bPhi_1(\var{1})=& \barynum{\var{1}}\otimes\bone_{\ladimi{2}}\otimes\bone_{\ladimi{3}}&\odot& \bX_1\otimes\bone_{\ladimi{2}}\otimes\bone_{\ladimi{3}}\\
\end{array}\right.,
$$
where $\bX_l=\left[\frac{1}{\var{l}-\lani{l}{1}}~\cdots~ \frac{1}{\var{l}-\lani{l}{j_l}}\right]^\top\,\in\,\IC^\ladimi{1}[\var{l}]$, where interpolation points are $\lani{l}{j_l}$, and $j_l=1,\cdots,\ladimi{l}$. Then $\barynum{\var{l}},\baryden{\var{l}}\in\IC^{\ladimi{1}\ladimi{2}\ladimi{3}}$ are composed  of $\ladimi{1}\ladimi{2}$ normalized barycentric weight vectors (interpolation values)  of size $\ladimi{l}$, which  are associated with single-variable functions with right interpolation points $\lani{l}{j_l}$.

\subsection{The $\ord$-variable case}

Similarly, for the general $\ord$ variable case, with $\bone_{\ladimi{i}\cdots\ladimi{i+\ell}}=\bone_{\ladimi{i}}\otimes\cdots\otimes\bone_{\ladimi{i+\ell}}$, and $N=\ladimi{1}\cdots\ladimi{\ord}$, we have ($l=1,\cdots,\ord$):
\begin{equation}\label{eq:general}
\begin{array}{rcl}
\bPhi_l(\var{l})&=& \left(\baryden{\var{l}}\otimes\bone_{\ladimi{l+1}\cdots \ladimi{\ord}}\right)\odot \left(\bone_{\ladimi{1}\cdots\ladimi{l-1}}\otimes\bX_l\otimes\bone_{\ladimi{l+1}\cdots\ladimi{n}}\right)
~~\mbox{and}\\%[1mm]
\widehat \bPhi_l(\var{l})&=& \left(\barynum{\var{l}}\otimes\bone_{\ladimi{l+1}\cdots \ladimi{\ord}}\right)\odot \left(\bone_{\ladimi{1}\cdots\ladimi{l-1}}\otimes\bX_l\otimes\bone_{\ladimi{l+1}\cdots\ladimi{n}}\right)
\end{array}
\end{equation}
where $\barynum{\var{l}},\baryden{\var{l}}\in\IC^{\ladimi{1}\cdots\ladimi{\ord}}$ are composed  of $\ladimi{1}\cdots\ladimi{l-1}$ normalized barycentric weight vectors (interpolation values)  of size $\ladimi{l}$, which  are associated with single-variable functions with right interpolation points $\lani{l}{j_l}$ with $l=1,\cdots,\ord$,  $j_l=1,\cdots,\ladimi{l}$. Finally, $\bX_l$ is the vector of normalized Lagrange polynomials in the variable $\var{l}$. Combining the expressions for the numerator and denominator, the following expression results:
\begin{equation}\label{mainformula}%\small
%\begin{array}{l}
\bH(\var{1},\cdots,\var{\ord})=
\frac{\displaystyle \sum_{\text{row-wise}}\widehat\bPhi_1(\var{1})
\odot\cdots\odot\widehat\bPhi_\ord(\var{\ord})}
{\displaystyle \sum_{\text{row-wise}}\bPhi_1(\var{1})
\odot\cdots\odot\bPhi_\ord(\var{\ord})}.
%~=~\frac{\displaystyle\sum_{k=1}^{\nu} ~ \mathstrut^1{\widehat\bPhi}_k(\mathstrut^1\bx) \cdots {\mathstrut^n\widehat\bPhi}_k(\mathstrut^n\bx)}{\displaystyle \sum_{k=1}^{\nu}~ \mathstrut^1\bPhi_k(\mathstrut^1\bx) \cdots \mathstrut^n\bPhi_k(\mathstrut^n\bx)} 
%\end{array}
\end{equation}

\section{Conclusions}\label{sec:conclusions}

In this note, we have addressed the issue of variable decoupling for rational multivariate functions, as well as numerator and denominator decoupling. This result complements the Kolmogorov Superposition Theorem, which addresses the issue of variable decoupling for arbitrary continuous functions. Making use of the results presented in \cite{AGPV25} and the connections with the Loewner framework, we show that the decoupling can be achieved by means of simple evaluations. No computation is required. In addition, the number of decoupled functions depends on the degrees of each variable. The following hold:
\begin{enumerate}
\item[\textbf{(a)}] We have thus shown that given a multivariate rational  function $\bH(\var{1},\cdots,\var{\ord})$, of degree $\ladimi{l}-1$ in the variable $\var{l}$, we can  write down the  decoupling (KST) functions in each  variable by inspection; no computation is required.
\item[\textbf{(b)}] This involves two important components: 
(i) writing down the null space of Loewner matrices associated with rational functions by inspection, and (ii) decoupling of the variables of multivariate rational functions  by evaluation.  The proof involves putting together the barycentric Lagrange weights  for each one of the single-variable functions and combining them using Kronecker and Hadamard products.
\item[\textbf{(c)}] Approximate decoupling can be used in the case of approximation of  non-rational multivariate functions. For details, we refer the reader to \cite{PVGVA}.
\end{enumerate}

%%%%%%%%%%%%%%%%%%%%%%%%%%%%%%%%%%%%%%%
\bibliographystyle{plain}
\bibliography{references}

\appendix

%%%%%%%%%%%%%%%%%%%%%%%%%%%%%%%%%%%%%%%%%%%%%%%%%%
\section{Recursive procedure: a two-variable rational example}\label{app:rec-rat}

%\section{Example: two-variable rational function decoupling} 

Consider the two-variable rational function  $\bH(\varS, \varT)=\frac{\num(\varS, \varT)}{\den(\varS, \varT)}$  with degrees one in $\varS$ and two in $\varT$, i.e. $\ladimi{1}=2$, $\ladimi{2}=3$. Let the right  interpolation points be $\varsi{1},\varsi{2}$ and $\varti{1},\varti{2},\varti{3}$.  The resulting (right) interpolation pairs are $(\varsi{i},\varti{j})$, $i=1,2$, and $j=1,2,3$. In other words, the interpolation grid, composed of six interpolation  pairs of points $(\varsi{i},\varti{j})$, enters into consideration. Variable decoupling is the consequence of  simple evaluations of the multivariate function $\bH(\varS,\varT)$  on the  above Kronecker grid, namely:
$$
\begin{array}{rcl}
\text{variable $\varS$ (one single-variable function)}&:&
\bH(\varS,\varti{3}):=\bphi(\varS)\\
\text{variable $\varT$ (two single-variable functions)}&:&
\bH(\varsi{i},\varT):=\bpsi_i(\varT) \text{ (for $i=1,2$)}, 
\end{array}
$$

Reconstruction of $\bH(\varS,\varT)$ is achieved from the three single-variable functions above. Here are the details (refer also to the first two lines of \eqref{NL3} define the normalized Lagrange basis vectors).

\subsection{The normalized barycentric coefficients and Lagrange basis: denominator}

According to Theorem \ref{thm:inspection_den}, we can write-down the barycentric coefficients  of the denominators of the two functions in the variable $\varT$ as follows (notice the normalization with respect to the last entry):
$$
\baryden{\varT}=\left[\begin{array}{c}
\baryden{\varti{1}}\\
\baryden{\varti{2}}\\
\end{array}\right]\in\IC^{6\times 1}.
$$
where
$$
\baryden{\varti{1}}=\left[\begin{array}{c}
\frac{\den(\varsi{1},\varti{1})}{\den(\varsi{1},\varti{3})}\\
\frac{\den(\varsi{1},\varti{2})}{\den(\varsi{1},\varti{3})}\\
1\\
\end{array}\right]\in\IC^{3\times 1}, ~~~
\baryden{\varti{2}}=\left[\begin{array}{c}
\frac{\den(\varsi{2},\varti{1})}{\den(\varsi{2},\varti{3})}\\
\frac{\den(\varsi{2},\varti{2})}{\den(\varsi{2},\varti{3})}\\
1\\
\end{array}\right]\in\IC^{3\times 1}.
$$

Then, for the variable $\varS$, the normalized barycentric weight vector is (notice that for the first variable, no normalization is applied):
$$
\baryden{\varS}=
\left[\begin{array}{c}
\den(\varsi{1},\varti{3})\\
\den(\varsi{1},\varti{3})
\end{array}\right].
$$
Thus, we define two single-variable vector functions, one is $\varS$ and one in $\varT$ as:
$$
\bPsi(\varT)=\baryden{\varT}\odot\left[\bI_2\otimes\bL^\vart(\varT)\right]~~\mbox{and}~~
\bPhi(\varS)=\left(\baryden{\varS}\odot\bL^\vars(\varS)\right)\otimes\bI_{3},
$$ 
being the two decoupled single-variable (vector) 
functions of size $N=6$  associated the denominator of $\bH(\varS,\varT)$.%: ${\bf Den}=\left[\begin{array}{c|c} \bPsi(t)~&~\bPhi(s) \end{array}\right]$.

\subsection{The normalized barycentric coefficients and Lagrange basis: numerator}

Similarly, we obtain the single-variable functions for the numerator:
$$
\barynum{\varT}=\left[\begin{array}{c}
\barynum{\varti{1}}\\
\barynum{\varti{2}}\\
\end{array}\right]\in\IC^{6\times 1}.
$$
where
$$
\barynum{\varti{1}}=\left[\begin{array}{c}
\frac{\num(\varsi{1},\varti{1})}{\num(\varsi{1},\varti{3})}\\
\frac{\num(\varsi{1},\varti{2})}{\num(\varsi{1},\varti{3})}\\
1\\
\end{array}\right]\in\IC^{3\times 1}, ~~~
\barynum{\varti{2}}=\left[\begin{array}{c}
\frac{\num(\varsi{2},\varti{1})}{\num(\varsi{2},\varti{3})}\\
\frac{\num(\varsi{2},\varti{2})}{\num(\varsi{2},\varti{3})}\\
1\\
\end{array}\right]\in\IC^{3\times 1}
$$
and
$$
\barynum{\varS}=
\left[\begin{array}{c}
\num(\varsi{1},\varti{3})\\
\num(\varsi{1},\varti{3})\end{array}\right].
$$
Thus, we define two single-variable vector functions, one is $\varS$ and one in $\varT$ as:
$$
\widehat{\bPsi}(\varT)=\barynum{\varT}\odot\left[\bI_2\otimes\bL^\vart(\varT)\right]~~\mbox{and}~~
\widehat{\bPhi}(\varS)=\left(\barynum{\varS}\odot\bL^\vars(\varS)\right)\otimes\bI_{3},
$$ 
being the two decoupled single-variable (vector) functions of size $N=6$ associated to the numerator of $\bH(\varS,\varT)$. %are:\\[1mm] ~${\bf Num}=\left[\begin{array}{c|c} \widehat\bPsi(t)~&~\widehat\bPhi(s) \end{array}\right]$.~

\subsection{Reconstruction of the rational function}

Combining the expressions for the numerator and denominator implies  that the original function, can be reconstructed as follows:
$$ %\label{2Dcase}
\bH(\varS,\varT)=\frac{\sum_{\textrm row-wise}\widehat{\bPsi}(\varT)\odot\widehat{\bPhi}(\varS)}
{\sum_{\textrm row-wise} \bPsi(\varT)\odot\bPhi(\varS)} .
%= \frac{\sum_{k=1}^{\nu}  \widehat{\bPsi}_k(\varT) \widehat{\bPhi}_k(\varS)}{\sum_{k=1}^{\nu} \bPsi_k(\varT) \bPhi_k(\varS)}.
$$

\subsection{Detailed barycentric representation and example}

We now provide a detailed barycentric representation of the numerator and denominator of $\bH(\varS,\varT)$ exhibiting the decoupling of the variables: 
$$
%\begin{turn}{90}{\!\!\!\!\!\!\mbox{\bf NUM (s,t)}}\end{turn}
\overbrace{\left[\begin{array}{c}
\frac{\num(\varsi{1},\varti{1})}{\num(\varsi{1},\varti{3})}\cdot \lagn{\varT}{\vart}{1} \\
\frac{\num(\varsi{1},\varti{2})}{\num(\varsi{1},\varti{3})}\cdot \lagn{\varT}{\vart}{2}\\
\varti{3}\\\hline
\frac{\num(\varsi{2},\varti{1})}{\num(\varsi{2},\varti{3})}\cdot \lagn{\varT}{\vart}{1}\\
\frac{\num(\varsi{2},\varti{2})}{\num(\varsi{2},\varti{3})}\cdot \lagn{\varT}{\vart}{2}\\
\varti{3}\\
\end{array}\right]}^{\widehat\bPsi(\varT)}
%%%%%%%
\odot
\overbrace{\left[ \begin{array}{c}
\num(\varsi{1},\varti{3})\cdot \lagn{\varS}{\vars}{1} \\~~\star~~\\~~\star~~\\
\hline
\num(\varsi{2},\varti{3})\cdot \lagn{\varS}{\vars}{2} \\~~\star~~\\~~\star~~\\\end{array} \right]}^{\widehat\bPhi(\varS)}=  
%%%%%%%%
\left[\begin{array}{r}
\num(\varsi{1},\varti{1})\cdot \lagn{\varS}{\vars}{1} \lagn{\varT}{\vart}{1} \\
\num(\varsi{1},\varti{2})\cdot\lagn{\varS}{\vars}{1} \lagn{\varT}{\vart}{2}\\
\num(\varsi{1},\varti{3})\cdot\lagn{\varS}{\vars}{1} \lagn{\varT}{\vart}{3}\\
\hline
\num(\varsi{2},\varti{1})\cdot\lagn{\varS}{\vars}{2} \lagn{\varT}{\vart}{1}\\
\num(\varsi{2},\varti{2})\cdot\lagn{\varS}{\vars}{2} \lagn{\varT}{\vart}{2}\\
\num(\varsi{2},\varti{3})\cdot\lagn{\varS}{\vars}{2} \lagn{\varT}{\vart}{3}\\
\end{array}\right]
$$
and
$$
%\begin{turn}{90}{\!\!\!\!\!\!\mbox{\bf NUM (s,t)}}\end{turn}
\overbrace{\left[\begin{array}{c}
\frac{ \den(\varsi{1},\varti{1})}{ \den(\varsi{1},\varti{3})}\cdot \lagn{\varT}{\vart}{1} \\
\frac{ \den(\varsi{1},\varti{2})}{ \den(\varsi{1},\varti{3})}\cdot \lagn{\varT}{\vart}{2}\\
\varti{3}\\\hline
\frac{ \den(\varsi{2},\varti{1})}{ \den(\varsi{2},\varti{3})}\cdot \lagn{\varT}{\vart}{1}\\
\frac{ \den(\varsi{2},\varti{2})}{ \den(\varsi{2},\varti{3})}\cdot \lagn{\varT}{\vart}{2}\\
\varti{3}\\
\end{array}\right]}^{\bPsi(\varT)}
%%%%%%%
\odot
\overbrace{\left[ \begin{array}{c}
\den(\varsi{1},\varti{3})\cdot \lagn{\varS}{\vars}{1} \\~~\star~~\\~~\star~~\\
\hline
\den(\varsi{2},\varti{3})\cdot \lagn{\varS}{\vars}{2} \\~~\star~~\\~~\star~~\\\end{array} \right]}^{\bPhi(\varS)}=  
%%%%%%%%
\left[\begin{array}{r}
\den(\varsi{1},\varti{1})\cdot \lagn{\varS}{\vars}{1} \lagn{\varT}{\vart}{1} \\
\den(\varsi{1},\varti{2})\cdot\lagn{\varS}{\vars}{1} \lagn{\varT}{\vart}{2}\\
\den(\varsi{1},\varti{3})\cdot\lagn{\varS}{\vars}{1} \lagn{\varT}{\vart}{3}\\
\hline
\den(\varsi{2},\varti{1})\cdot\lagn{\varS}{\vars}{2} \lagn{\varT}{\vart}{1}\\
\den(\varsi{2},\varti{2})\cdot\lagn{\varS}{\vars}{2} \lagn{\varT}{\vart}{2}\\
\den(\varsi{2},\varti{3})\cdot\lagn{\varS}{\vars}{2} \lagn{\varT}{\vart}{3}\\
\end{array}\right].
$$

%%%%%%%%%%%%%%%%%%%%%%%
Let us now consider the example: $\bH(\varS,\varT)=\frac{\varT^2  +\varS-2}{\varT^2+2\varS+1}$, together with the  interpolation points: $(\varsi{1},\varsi{2})$ and $(\varti{1},\varti{2},\varti{3})$. This example leads to:
$$
\begin{array}{c}
\bPhi(\varS)\odot \bPsi(\varT) \\ \hline
( t_1^2 +2 s_1 + 1)\lagn{\varS}{\vars}{1} \lagn{\varT}{\vart}{1}\\
( t_2^2+2s_1+1)\lagn{\varS}{\vars}{1} \lagn{\varT}{\vart}{2}\\
( t_3^2 + 2s_1 + 1)\lagn{\varS}{\vars}{1} \lagn{\varT}{\vart}{3}\\
( t_1^2+2s_2 + 1)\lagn{\varS}{\vars}{2} \lagn{\varT}{\vart}{1}\\
( t_2^2+2s_2+ 1)\lagn{\varS}{\vars}{2} \lagn{\varT}{\vart}{2}\\
( t_3^2 + 2s_2 + 1)\lagn{\varS}{\vars}{2} \lagn{\varT}{\vart}{3}\\
\end{array}~ \text{ and }
\begin{array}{c}
\widehat\bPhi(\varS)\odot \widehat\bPsi(\varT) \\ \hline
(  t_1^2 + s_1 - 2)\lagn{\varS}{\vars}{1} \lagn{\varT}{\vart}{1}\\
(  t_2^2 + s_1 - 2)\lagn{\varS}{\vars}{1} \lagn{\varT}{\vart}{2}\\
(  t_3^2 + s_1 - 2)\lagn{\varS}{\vars}{1} \lagn{\varT}{\vart}{3}\\
(  t_1^2 + s_2 - 2)\lagn{\varS}{\vars}{2} \lagn{\varT}{\vart}{1}\\
(  t_2^2 + s_2 - 2)\lagn{\varS}{\vars}{2} \lagn{\varT}{\vart}{2}\\
(  t_3^2 + s_2 - 2)\lagn{\varS}{\vars}{2} \lagn{\varT}{\vart}{3}\\
\end{array}.
$$
By summing along the rows,
$$
\left\{\begin{array}{l}
\hat\num(\varS,\varT)=\sum_{\textrm row-wise}\widehat\bPhi(\varS)\odot \widehat\bPsi(\varT) = ( \varT^2 + \varS - 2)\\
\hat\den(\varS,\varT) = \sum_{\textrm row-wise}\bPhi(\varS)\odot\bPsi(\varT)  = ( \varT^2 + 2\varS + 1)\\
\end{array} 
\right.
$$
Then, 
$$
\frac{\hat\num(\varS,\varT)}{\hat\den(\varS,\varT)} =
\frac{\varT^2  +\varS-2}{\varT^2+2\varS+1}=\bH(\varS,\varT).
$$

%%%%%%%%%%%%%%
%\section{The P\'olya-Szeg\"o example (cont'd)}\label{app:polya}

\end{document}